\documentclass[10pt,twoside,a4paper,english,reqno]{amsart}
\usepackage{listings,graphicx,amsmath,varioref,amscd,amssymb,color,bm}

\usepackage[colorlinks=printing,backref]{hyperref}      

\usepackage{upref}

\newtheorem{defi}{Definition}[section]
\newtheorem{theo}{Theorem}[section]
\newtheorem{lemm}{Lemma}[section]

\newcommand{\R}{\mathbb{R}}
\newcommand{\N}{\mathbb{N}}
\newcommand{\Div}{\mathrm{div}}

\newcommand{\wh}{\hat{w}}
\newcommand{\xh}{\hat{z}}
\newcommand{\yo}{s}

\newcommand{\CLea}{\mathcal{L}_\eps}
\newcommand{\Hneg}{H_{\mathrm{loc}}^{-1}}
\newcommand{\byo}{{\underline{s}}}
\newcommand{\byh}{{\underline{\hat{y}}}}
\newcommand{\bwh}{{\underline{\hat{w}}}}
\newcommand{\tf}{\tilde{f}}

\newcommand{\bw}{{\underline{w}}}
\newcommand{\Ly}{{L^\infty}}
\newcommand{\un}{\mathbf{1}}
\newcommand{\bA}{{A_\alpha}}
\newcommand{\bH}{{q_\alpha}}
\newcommand{\chif}{$\chi$-function}
\newcommand{\Open}{\mathcal{O}}
\newcommand{\ah}{\hat{a}}

\newcommand{\ftau}{{f^{\tau}}}
\newcommand{\tao}{\tilde{a}^0}

\newcommand{\tm}{\tilde{m}}
\newcommand{\tme}{\tilde{m}_\delta}
\newcommand{\tmen}{\tilde{m}_{\delta_n}}
\newcommand{\tfe}{\tilde{f}_\delta}

\newcommand{\tfen}{\tilde{f}_{\delta_n}}
\newcommand{\tfi}{\tilde{f}_\infty}

\newcommand{\esslim}{\operatorname*{ess\,lim}}
\newcommand{\eps}{\varepsilon}
\newcommand{\norm}[1]{\left\Vert#1\right\Vert}
\newcommand{\abs}[1]{\left|#1\right|}
\newcommand{\Set}[1]{\left\{#1\right\}}
\newcommand{\Lenloc}{L^1_{\mathrm{loc}}}

\newcommand{\sgn}{\operatorname*{sgn}}
\newcommand{\pt}{\partial_t}
\newcommand{\px}{\partial_x }
\newcommand{\pxx}{\partial_{xx}^2}
\newcommand{\pxxx}{\partial_{xxx}^3}
\newcommand{\pxxxx}{\partial_{xxxx}^4}

\newcommand{\ptx}{\partial_{tx}^2}
\newcommand{\ptxx}{\partial_{txx}^3}
\newcommand{\ue}{u_\eps}

\newcommand{\Pe}{P_\eps}
\newcommand{\Ve}{V_\eps}
\newcommand{\Pek}{P_{\eps_k}}
\newcommand{\ome}{\omega_\eps}
\newcommand{\ve}{v_\eps}
\newcommand{\Ome}{\Omega_\eps}
\newcommand{\te}{\theta_\eps}
\newcommand{\uek}{u_{\eps_k}}

\title[IBVP for conservation laws with source terms and the D-P equation]
{Initial-boundary value problems for \\ conservation laws with source terms
\\ and the Degasperis-Procesi equation}

\author[G. M. Coclite]{G. M. Coclite}
\address[Giuseppe Maria Coclite]{\newline Dipartimento di Matematica \newline
Universit\`a degli Studi di Bari \newline
Via E. Orabona 4\newline
70125 Bari, Italy}
\email[]{coclitegm@dm.uniba.it}
\urladdr{http://www.dm.uniba.it/Members/coclitegm/}

\author[K. H. Karlsen]{K. H. Karlsen}
\address[Kenneth Hvistendahl Karlsen]
{\newline
Centre of Mathematics for Applications \newline
University of Oslo \newline
P.O. Box 1053, Blindern \newline
N--0316 Oslo, Norway}
\email[]{kennethk@math.uio.no}
\urladdr{http://folk.uio.no/kennethk/}

\author[Y.-S. Kwon]{Y.-S. Kwon}
\address[Young-Sam Kwon]
{\newline
CSCAMM, 4149 CSIC Building \#406\newline
Paint Branch Drive\newline
University of Maryland\newline
College Park, MD 20742-2389}
\email[]{ykwon@cscamm.umd.edu}



\thanks{The research of K. H. Karlsen is supported by an Outstanding
Young Investigators Award (OYIA) from the Research Council of
Norway. This work was initiated while K. H. Karlsen visited CSCAMM
at the University of Maryland. He is grateful for CSCAMM's financial
support and excellent working environment. This article was written as
part of the the international research program
on Nonlinear Partial Differential Equations at the Centre for
Advanced Study at the Norwegian Academy of Science
and Letters in Oslo during the academic year 2008--09. }

\date{\today}

\begin{document}

\maketitle

\begin{abstract}
We consider conservation laws with source terms
in a bounded domain with Dirichlet boundary conditions.
We first prove the existence of a strong trace at the boundary
in order to provide a simple  formulation of the entropy
boundary condition. Equipped with this formulation, we go on to establish the
well-posedness of entropy solutions to the initial-boundary value problem.
The proof utilizes the kinetic formulation and the
compensated compactness method. Finally, we make use of these
results to demonstrate the well-posedness in a class of
discontinuous solutions to the initial-boundary value problem for the
Degasperis-Procesi shallow water equation, which
is a third order nonlinear dispersive equation
that can be rewritten in the form of a nonlinear conservation law with
a nonlocal source term.
\end{abstract}

\tableofcontents

\section{Introduction}
In this article we consider scalar conservation laws with source
terms on a bounded open subset $\Omega\subset\R^d$ with
$C^2$ boundary:
\begin{equation}\label{eq_moment}
    \partial_t u+\Div_x A(u)=S(t,x,u),
    \quad  (t,x)\in Q:=(0,T)\times\Omega,
\end{equation}
where $T>0$ is a fixed final time and
the flux function $A\in C^2$ satisfies the
genuine nonlinearity condition
\begin{equation}\label{hypothesis}
    \mathcal{L}(\{\xi \ | \ \tau+\zeta\cdot A'(\xi)=0\})=0,\quad
    \mathrm{for \ every \ \ }(\tau,\xi)\neq(0,0),
\end{equation}
where $\mathcal{L}$ is the Lebesgue measure.

The source term satisfies the following conditions:
\begin{equation}\label{eq:source-ass}
    S\in L^\infty(Q\times\R), \quad
    S(t,x,\cdot)\in  C^1(\R), \quad
    \abs{S(t,x,u)-S(t,x,v)} \le C|u-v|,
\end{equation}
where the last two conditions hold
for a.e.~$(t,x)\in Q$ and $C>0$ is a constant.

As usual, we only deal  with entropy solutions, namely
those that fulfill in the sense of distributions on $Q$ the inequality
\begin{equation}\label{ineq_entropie}
    \partial_t\eta(u)+\Div_x q(u)-\eta'(u)S(t,x,u)\le 0
 \end{equation}
for every convex $C^2$ function $\eta$ and related entropy flux defined by
\begin{equation*}
    q'=A'\eta'.
\end{equation*}
We are interested in the well-posedness in $L^\infty$
of the initial-boundary value problem for \eqref{eq_moment},
in which case we impose the initial data
\begin{equation}\label{eq:initial}
    u(0,\cdot)=u_0\in L^\infty(\Omega)
\end{equation}
and the Dirichlet boundary data
\begin{equation}\label{eq:boundary}
    u|_{\Gamma}=u_b\in L^\infty(\Gamma),
\end{equation}
where $\Gamma:=(0,T)\times\partial\Omega$. Of course, this Dirichlet condition has to
be interpreted in an appropriate sense (see below) and this in turn requires an entropy
solution to possess boundary traces (which herein will be understood in a strong sense).

A $BV$ well-posedness theory for conservation laws with Dirichlet boundary conditions was
first established by Bardos, le Roux, and N{\'e}d{\'e}lec \cite{Bardos:1979us}, and later
extended by Otto \cite{Otto:1996am} to the $L^\infty$ setting, for which boundary traces do
not exist in general, a fact that complicates significantly the notion of solution and the proofs.
For genuinely nonlinear fluxes and domains whose boundaries satisfy a
mild regularity assumption, Vasseur \cite{Vasseur} showed
that $L^\infty$ entropy solutions always have traces at the boundaries.
Similar results hold without imposing a genuine nonlinearity
condition, cf.~Panov \cite{Panov1,Panov2} and Kwon and Vasseur \cite{K-V}.
Consequently, for genuinely nonlinear fluxes, the $L^\infty$ case
can be treated as in \cite{Bardos:1979us}, i.e.,
the more complicated notion of entropy solution
used by Otto can be avoided, see Kwon \cite{Kwon}.

To define traces on the boundary $\Gamma$
we use the concept of a ``regular deformable boundary" (see for
instance Chen and Frid in \cite{C-F}).
For any domain $\Omega$ with ${\mathcal{C}}^2$ boundary, there
exists at least one $\partial\Omega$-regular deformation. Given any
open subset $\hat{K}$ of $\partial\Omega$, we refer to a mapping
$\hat{\psi}:[0,1]\times \hat{K}\to\bar{\Omega}$ as a $\hat{K}$-regular
deformation provided it is a ${\mathcal C}^1$ diffeomorphism
and $\hat{\psi}(0,\cdot)\equiv I_{\hat{K}}$ with $I_{\hat{K}}$ denoting the
identity map over $\hat{K}$. Let us now define the
set $K:=(0,T)\times \hat{K}$ and the function
$\psi(t,x):=(t,\hat{\psi}(x))$. Then, obviously, $\psi(t,x)$ is $K$-regular
deformation with respect to $\Gamma$.
Let us denote by $\hat{n}_s$ the unit outward normal
field of the deformed boundary $\hat{\psi}(\{s\}\times\partial \Omega)$. We
also write $n_s=(0,\hat{n}_s)$ and $n=(0,\hat{n})$. Notice that $\hat{n}_s$ converges
strongly to $\hat{n}$ when $s$ goes to 0.

Our first main result is the following theorem.

\begin{theo}\label{the_theorem1}
Let $\Omega\subset\R^d$ be a regular open set with $\mathcal{C}^2$ boundary. Assume
that \eqref{eq:source-ass} holds and that the flux
function $A\in C^{2}(\R)$ verifies \eqref{hypothesis}.
Consider any function $u\in L^\infty((0,T)\times\Omega)$
obeying \eqref{eq_moment} and \eqref{ineq_entropie}
in $(0,T)\times\Omega$. Then

$\bullet$ there exists $u^\tau\in L^\infty((0,T)\times\partial \Omega)$
such that for  every $\Gamma$-regular deformation $\psi$
and every compact set $K\subset\!\subset \Gamma$ there holds
\begin{equation*}
    \esslim_{s\to0} \int_{K}|
    u(\psi(s,\hat{z}))-u^\tau(\hat{z})|d\sigma(\hat{z})=0,
\end{equation*}
where $d\sigma$ denotes the volume
element of $(0,T)\times\partial\Omega$;

$\bullet$ there exists $u^\tau\in L^\infty(\Omega)$ such
that for every compact set $K\subset\!\subset\Omega$ there holds
\begin{equation*}
    \esslim_{t\to0} \int_{K}
    |u(t,x)-u^\tau(x)|\,dx=0.
\end{equation*}
In particular, the trace $u^\tau$ is unique and, for any
continuous function $F$, $F(u)$ also possesses a trace and
\begin{equation*}
    [F(u)]^\tau=F(u^\tau).
\end{equation*}
\end{theo}
The proof of this theorem is found in Section \ref{sec:strong_trace}.
More precisely, in that section we prove the first part
of Theorem \ref{the_theorem1}. The second part can be
proved using the same method, so we omit  the details.

Having settled the existence of strong boundary traces, we can
now turn to the choice of entropy boundary condition.
Instead of working with the original condition
due to Bardos, le Roux, and N{\'e}d{\'e}lec \cite{Bardos:1979us}, we
shall instead employ the following equivalent boundary
condition introduced by Dubois and LeFloch \cite{D-L}, which
is well-defined in $L^\infty$ thanks to Theorem \ref{the_theorem1}:
\begin{equation}\label{DL_boundary}
    \Bigl[q(u^\tau)
    -q(u_b)-\eta'(u_b)
    (A(u^\tau)-A(u_b))\Bigr]\cdot\hat{n}\geq 0,
\end{equation}
where $B^\tau$ means the trace
of $B$ on $\Gamma=(0,T)\times\partial\Omega$
and  $\hat{n}$ is the unit outward normal
to $\partial\Omega$ with $n=(0,\hat{n})$.

Our second main result is the well-posedness of entropy solutions to
the initial-boundary value problem \eqref{eq_moment}, \eqref{eq:initial},
and \eqref{eq:boundary}, with the boundary condition
\eqref{eq:boundary} being interpreted in the sense of \eqref{DL_boundary}.

\begin{theo}\label{the_theorem2}
Let $\Omega\subset\R^d$ be a regular open set with $\mathcal{C}^2$
boundary. Assume that the source term $S(t,x,u)$ obeys \eqref{eq:source-ass}
and that the flux function $A\in C^2(\R)$ verifies \eqref{hypothesis}. Let
$u_0\in L^\infty(\Omega)$. Then there exists an unique entropy
solution $u\in L^\infty(Q)$ verifying \eqref{eq_moment},
\eqref{ineq_entropie}, \eqref{eq:initial}, and \eqref{DL_boundary}.
\end{theo}

This theorem is proved in Section \ref{sec:proof-existence-uniqueness}.
As in \cite{Kwon}, the uniqueness argument utilizes the Dubois and LeFloch boundary
condition \eqref{DL_boundary} written in a kinetic form, which plays an essential role in
the proof of uniqueness.

In Section \ref{sec:DP} we apply Theorems \ref{the_theorem1} and \ref{the_theorem2} to
investigate the well-posedness of the initial-boundary value
problem for the so-called \textit{Degasperis-Procesi} equation
\begin{equation}
    \label{eq:DP}
    \pt  u-\ptxx u+4u\px u=3\px u\pxx u +u\pxxx u,
    \qquad (t,x)\in (0,T)\times (0,1),
\end{equation}
augmented with the initial condition
\begin{equation}\label{eq:init}
    u(0,x)=u_0(x),\qquad x\in (0,1),
\end{equation}
and the boundary data
\begin{equation}\label{eq:DPboundary}
    \begin{split}
        u(t,0)=g_0(t),&\qquad u(t,1)=g_1(t),\quad t\in(0,T),\\
        \px u(t,0)=h_0(t),&\qquad \px u(t,1)=h_1(t),\quad t\in(0,T).
    \end{split}
\end{equation}

We assume that
\begin{equation}\label{eq:ass}
	\begin{split}
		&u_0\in L^\infty(0,1),\qquad u_0(0)=g_0(0),\qquad u_0(1)=g_1(0),\\
		&g_0,\,g_1\in H^1(0,T),\quad h_0,\,h_1\in L^\infty(0,T).
	\end{split}
\end{equation}

Degasperis and Procesi \cite{DP:99} deduced \eqref{eq:DP} from the
following family of third order dispersive
nonlinear equations, indexed
over six constants $\alpha,\gamma,c_0,c_1,c_2,c_3\in \R$:
$$
\pt u + c_0 \px u + \gamma \pxxx u - \alpha^2 \ptxx u
= \px \left(c_1 u^2 + c_2 (\px u)^2 + c_3 u\pxx u\right).
$$
Using the method of asymptotic integrability, they
found that only three equations within this family
were asymptotically integrable up to the third order: the
\textit{KdV equation} ($\alpha=c_2=c_3=0$), the
\textit{Camassa-Holm equation} ($c_1=-\frac{3c_3}{2\alpha^2}$, $c_2=\frac{c_3}{2}$), and one
new equation ($c_1=-\frac{2c_3}{\alpha^2}$, $c_2=c_3$),
which properly scaled reads
\begin{equation}\label{eq:almostDP}
    \pt u + \px u + 6 u \px u + \pxxx u
    - \alpha^2\left(\ptxx u + \frac{9}{2}\px u \pxx u + \frac{3}{2}u \pxxx u\right)
    =0.
\end{equation}
By rescaling, shifting the dependent variable, and finally applying a
Galilean boost, equation \eqref{eq:almostDP} can be transformed into
the form \eqref{eq:DP}, see \cite{DHH:2003,DKK:2002} for details.

Degasperis, Holm, and Hone \cite{DKK:2002}  proved the
integrability of \eqref{eq:DP} by constructing a Lax pair. Moreover,
they provided a relation to a negative flow in the Kaup-Kupershmidt
hierarchy by a reciprocal transformation and derived two infinite
sequences of conserved quantities along with a bi-Hamiltonian
structure. Furthermore, they showed that the Degasperis-Procesi equation
are endowed with weak (continuous) solutions that are superpositions of
multipeakons and described the integrable finite-dimensional peakon
dynamics. An explicit solution was also found in the
perfectly anti-symmetric peakon-antipeakon collision case. Lundmark
and Szmigielski \cite{LS:2003}, using an inverse scattering
approach, computed $n$-peakon solutions to \eqref{eq:DP}.
Mustafa \cite{Mustafa:DP05} proved that smooth solutions to
\eqref{eq:DP} have infinite speed of propagation: they lose
instantly the property of having compact support.  Regarding  the
Cauchy problem for the Degasperis-Procesi equation  \eqref{eq:DP},
Escher, Liu, and Yin  have studied its well-posedness within certain
functional classes in a series of  papers \cite{ELY:2006, ELY:2007,
EY, LY:2006,Yin-DP:2003-JMAA,Yin-DP:2003-Illinois,Yin-DP:2004-Indiana,Yin-DP:2004-FunctAnal}.

The approach taken in the papers just listed emphasizes the similarities
between the Degasperis-Procesi equation and the Camassa-Holm equation, and
consequently the main focus has been on  (weak) continuous solutions.
In a rather different direction, Coclite and Karlsen \cite{Coclite:2005cr,CK:DPol,CK:DPUMI}
and Lundmark \cite{Lundmark:2005kz} initiated a study of
discontinuous (shock wave) solutions to the Degasperis-Procesi equation \eqref{eq:DP}.
In particular, the existence, uniqueness, and stability of entropy
solutions of the Cauchy problem for \eqref{eq:DP}
is proved in \cite{Coclite:2005cr,CK:DPol,CK:DPUMI}.

When it comes to initial-boundary value problems for
the Degasperis-Procesi equation much less is known.
The first results in that direction are
those of Escher and Yin \cite{EY, Z}, which apply to
continuous solutions.

To encompass discontinuous solutions we shall herein
extend the approach of \cite{Coclite:2005cr,CK:DPol,CK:DPUMI}, relying on
Theorems \ref{the_theorem1} and \ref{the_theorem2} above.
Following \cite{Coclite:2005cr} we rewrite \eqref{eq:DP},
\eqref{eq:init},  \eqref{eq:DPboundary} as a hyperbolic-elliptic system with
boundary conditions:
\begin{equation}\label{eq:DPw}
    \begin{cases}
        \pt u+u \px u+\px P= 0,&\quad (t,x)\in (0,T)\times(0,1),\\
        -\pxx P+P=\frac{3}{2}u^2,&\quad (t,x)\in (0,T)\times(0,1),\\
        u(0,x)=u_{0}(x),&\quad x\in (0,1),\\
        u(t,0)=g_{0}(t),\>\>\> u(t,1)=g_{1}(t),&\quad t\in(0,T),\\
        \px P(t,0)=\psi_{0}(t),\>\>\>  \px P(t,1)=\psi_{1}(t),&\quad t\in(0,T),
    \end{cases}
\end{equation}
where
\begin{equation}\label{eq:DPboundaryP}
    \psi_{0}=-g_{0}'- g_{0}h_{0},\qquad \psi_{1}=-g_{1}'- g_{1}h_{1}.
\end{equation}
Indeed, formally, from \eqref{eq:DP},
\begin{equation}\label{eq:DPw1}
    (1-\pxx)(\pt u+ u \px u+\px P )=0,
\end{equation}
since, by \eqref{eq:DPboundaryP}, the trace of
$\pt u+ u \px u+\px P$ vanishes at $x=0$ and $x=1,$ we
can invert the differential operator
$1-\pxx$ and pass from \eqref{eq:DPw1} to \eqref{eq:DPw}.

In the case $g_0=g_1=0$ we do not need any
boundary condition on $\px u$, indeed from \eqref{eq:DPboundaryP}
we have $\psi_0=\psi_1=0$.

The boundary conditions related to the $P$-equation
in \eqref{eq:DPw} are of Neumann type.
Let $\widetilde G=\widetilde G(x,y)$ be
the Green's function of the operator $1-\pxx$ with Neumann
boundary $\psi_0,\,\psi_1$ conditions on $(0,1)$.
The function $P$ has a convolution structure
\begin{equation}\label{eq:def-Pu}
    P(t,x)=P^u(t,x):=\frac32 \int_0^1 \widetilde G(x,y) u^2(t,y)\,dy,
\end{equation}
and \eqref{eq:DPw} can be written as a
conservation law with a nonlocal source
\begin{equation*}
    \pt  u+\px \left( \frac{u^2}{2}\right)=-\px P^u
    =- \frac32 \int_0^1 \px \widetilde G(x,y) u^2(t,y)\,dy.
\end{equation*}

Due to the regularizing effect of the elliptic equation in
\eqref{eq:DPw} we have that
\begin{equation}\label{eq:DPsmooth}
    u\in L^\infty((0,T)\times(0,1))\Longrightarrow
    P^u\in L^{\infty}(0,T;W^{2,\infty}(0,1)).
\end{equation}
Therefore, if a map $u\in L^\infty((0,T)\times(0,1))$ satisfies, for every convex
map $\eta\in  C^2$,
\begin{equation}
    \label{eq:DPentropy}
    \pt \eta(u)+ \px q(u)+\eta'(u)\px P^u\le 0, \qquad
    q(u)=\int^u \xi \eta'(\xi)\, d\xi,
\end{equation}
in the sense of distributions, then Theorem \ref{the_theorem1}
provides the existence of strong traces $u^\tau_0,\, u^\tau_1$ on the
boundaries $x=0,1$, respectively.

We say that  $u\in L^\infty((0,T)\times(0,1))$ is an entropy solution of the initial-boundary
value problem \eqref{eq:DP}, \eqref{eq:init}, \eqref{eq:DPboundary} if

(i) $u$ is a distributional solution of \eqref{eq:DPw};

(ii) for every convex function $\eta\in  C^2(\R)$ the
entropy inequality \eqref{eq:DPentropy} holds in the sense of distributions;

(iii) for every convex function $\eta\in  C^2$ with corresponding $q$ defined by $q'(u)=u\eta'(u)$, 
the boundary entropy condition
\begin{equation}\label{eq:DPentropyboundary}
	\begin{split}
		& q(u^\tau_0(t))-q(g_0(t))-\eta'(g_0(t))\frac{(u_0^\tau(t))^2-(g_0(t))^2}{2}
		\\ & \le 0\le 
		q(u^\tau_1(t))-q(g_1(t))-\eta'(g_1(t))\frac{(u_1^\tau(t))^2-(g_1(t))^2}{2}
	\end{split}
\end{equation}
holds for a.e.~$t\in(0,T)$.

Our main result for the initial-boundary value problem for the Degasperis-Procesi equation
is the following theorem, which is proved in Section \ref{sec:DP}.
\begin{theo}\label{the_theorem3}
Let $u_0,\,\gamma,\,g_0,\,g_1,\,h_0,\,h_1$ satisfy \eqref{eq:ass}.
The initial-boundary value problem
\eqref{eq:DP}, \eqref{eq:init}, \eqref{eq:DPboundary} possesses
an unique entropy solution $u\in L^\infty((0,T)\times(0,1))$.
\end{theo}

\section{Proof of Theorem \ref{the_theorem1}}\label{sec:strong_trace}

\subsection{Weak boundary trace}
We first reformulate the relevant problems on local open subsets and
construct weak boundary traces of entropy solutions on these local sets.
The reason for working on local subsets is that we are going to use the blow-up method.
We split the boundary into a countable number of subsets.
Indeed, for each $\hat{x}\in\partial\Omega$, there exists $r_{\hat{x}}>0$, a $C^2$ mapping
$\gamma_{\hat{x}}:\R^{d-1}\to\R^{d-1}$, and an isometry for the
Euclidean norm $\mathcal{R}_{\hat{x}}:\R^{d}\to\R^{d}$ such that,
upon rotating, relabeling, and translating the coordinate axes if necesary,
\begin{align*}
    &\mathcal{R}_{\hat{x}}(\hat{x})=0, \\
    & \mathcal{R}_{\hat{x}}(\Omega)\cap (-r_{\hat{x}},r_{\hat{x}})^{d}
    =\{y=(y_0,\hat{y})\in(-r_{\hat{x}},r_{\hat{x}})^{d} \ | \
    y_0 >\gamma_{\hat{z}}(\hat{y})\}.
\end{align*}
We have
$$
\partial\Omega\subset \bigcup_{\hat{x}\in\partial\Omega}
\mathcal{R}_{\hat{x}}^{-1}((-r_{\hat{x}},r_{\hat{x}})^{d}).
$$
Hence, for each $\hat{z}=(\hat{t},\hat{x})\in\Gamma$, we obtain an
isometry map $\Lambda_{\hat{z}}:\R^{d+1}\to\R^{d+1}$ given by
$\Lambda_{\hat{z}}(t,x)=(y_0,t-\hat{t},\hat{y})$, where
$(y_0,\hat{y})=\mathcal{R}_{\hat{x}}(x)$. Then we have
$$
\Gamma=\bigcup_{\hat{z}\in \Gamma}
(\Lambda_{\hat{z}})^{-1}((0,r_{\hat{z}})
\times(-r_{\hat{z}},r_{\hat{z}})^d).
$$
Since the above collection of open sets is countable,
$$
\bigcup_{\hat{z}\in \Gamma}(\Lambda_{\hat{z}})^{-1}(\Gamma_{\hat{z}})
=\bigcup_{\alpha\in K}(\Lambda_{\alpha}) ^{-1}(\Gamma_\alpha),
$$
where $K$ is a countable set and
$$
\Gamma_\alpha=\Lambda_\alpha^{-1}(\{w=(w_0,\hat{w})\in
(0,T)\times(-r_\alpha,r_\alpha)^{d} \ | \
w_0=\Lambda_{\alpha}(\hat{w})\}),
$$
where $w=(w_0,\hat{w})=(y_0,t-\hat{t},\hat{y})$ and $\hat{w}=(t-\hat{t},\hat{y})$.
In an attempt to simplify the notation we
write $\alpha$ instead of $\hat{z}_\alpha$ in the indices.
We define
$$
Q_\alpha=\{w\in(0,r_\alpha)\times
(-r_\alpha,r_\alpha)^d \ | \ w_0>\Lambda_\alpha(\hat{w})\}.
$$
From now on we will work in $Q_\alpha$ and state the equations
in terms of the new $w$ variable.
To this end, define $u_\alpha:Q_\alpha\to\R$
by $u_\alpha(w)=u((\Lambda_\alpha)^{-1}(w))$ and set
$A_\alpha(\xi)=\Lambda_{\alpha}(\xi,A(\xi))$,
$\bH(\xi)=\Lambda_\alpha(\eta(\xi),q(\xi))$.
For every fixed $\alpha$, every deformation $\psi$, and
every $\hat{w}\in (-r_\alpha, r_\alpha)^d$, we define
$$
\tilde{\psi}(s,\hat{w})=(\Lambda_{\alpha}\circ\psi)(s,(\mathcal{R}_{\alpha})^{-1}(\hat{w})),
\qquad s=w_0.
$$

In terms of the $w$ variable, \eqref{eq_moment}
and \eqref{ineq_entropie} read respectively
\begin{equation}\label{eq_moment'}
    \Div_w \bA(u_\alpha)=0 \quad \text{in $Q_\alpha$}
\end{equation}
and
\begin{equation}\label{ineq_entropie'}
    \Div_w \bH(u_\alpha)\le 0 \quad \text{in $Q_\alpha$}.
\end{equation}

We now introduce a kinetic formulation of \eqref{eq_moment'}
and \eqref{ineq_entropie'}, cf.~\cite{LPT}. To do so
we set $L=\|u\|_{L^\infty(\Omega)}$, bring in a
new variable $\xi\in(-L,L)$, and introduce
for every $v\in(-L,L)$ the function
$$
\chi(v,\xi) =
\begin{cases}
    \un_{\{0\le\xi\le v\}}, & \text{if $v\geq0$}, \\
    -\un_{\{v\le\xi\le0\}}, & \text{if $v<0$}.
\end{cases}
$$
To effectively represent weak limits of nonlinear functions
of weakly converging sequences,  we introduce new functions, called
microscopic functions, which depend on $\xi$ and on an
additional variable $z$ \cite{Perthame:2002qy}.

\begin{defi}\label{def_qifunction}
Let $N$ be an integer and $\Open$ be an open set of $\R^N$. We
say that $f\in \Ly(\Open\times (-L,L))$ is a microscopic function if it obeys
$0\le \mathrm{sgn}(\xi)f(z,\xi)\le 1$ for almost every
$(z,\xi)$. We say that $f$ is a \chif\ if there exists a function
$u\in L^\infty (\Open)$ such that for a.e.~$z\in\Open$ there holds
$f(z,\cdot)=\chi(u(z),\cdot)$.
\end{defi}

For later use, let us collect the following results (cf.~\cite{Perthame:2002qy}).
\begin{lemm}\label{elementaire}
Fix an open set ${\Open}\subset \R^N$, and let $f_k\in
\Ly({\Open}\times (-L,L))$ be a sequence of \chif s
$L^\infty_{\text{weak-$\star$}}$-converging to $f\in \Ly(\Open\times (-L,L))$.
Introduce the functions $u_k(\cdot)=\int_{-L}^L f_k(\cdot,\xi)\,d\xi$ and
$u(\cdot)=\int_{-L}^L f(\cdot,\xi)\,d\xi$. Then, for almost every
$z\in{\Open}$, the function $f(z,\cdot)$ lies in $BV(-L,L)$.
Moreover, the following statements are equivalent:
\begin{itemize}
    \item $f_k$ converges strongly to $f$ in $L^1_{\mathrm{loc}}(\Open\times (-L,L))$.
    \item $u_k$ converges strongly to $u$ in $L^1_{\mathrm{loc}}(\Open)$.
    \item $f$ is a \chif.
\end{itemize}
\end{lemm}

Observe that if $f$ is a \chif\, then $u(z)=\int_{-L}^L f(z,\xi)\,d\xi$.
The following theorem is due to Lions, Perthame, and Tadmor \cite{LPT}.
\begin{theo}\label{theo_LPT}
A function $u\in\Ly(Q_\alpha)$,  with $|u|\le L$,  is a solution
of \eqref{eq_moment'} and \eqref{ineq_entropie'} if and
only if there exists a nonnegative measure
$m\in \mathcal{M}^+({Q_\alpha}\times(-L,L))$ such that the related \chif\
$f$ defined by $f(u(w),\xi)=\chi(u(w),\xi)$ for almost every
$(w,\xi)\in(Q_\alpha\times(-L,L))$ verifies
\begin{equation}\label{eq_cinetique}
    a(\xi)\cdot\nabla_w f-S(\cdot,\xi)(\partial_\xi f-\delta(\xi))=\partial_\xi m
    \quad \text{in $\mathcal{D}'(Q_\alpha\times(-L,L))$,}
\end{equation}
where $a(\xi):=A'_\alpha(\xi)$.
\end{theo}
Denote $a$ by $a=(a_0,\hat{a})$. 
To simplify the notation we keep denoting the normal
vectors by $n_s$ and $n$.

In what follows, for every fixed $\alpha$, we will consider  the set $Q_\alpha$, and the
\chif\ $f$ associated to $u_\alpha$. For every  regular deformation
$\psi$ and every $\hat{w}\in(-r_\alpha, r_\alpha)^d$ we set:
\begin{eqnarray*}
&&\tilde{\psi}(s,\hat{w})=\psi(s,\Lambda_\alpha^{-1}
(\gamma_\alpha(\hat{w}),\hat{w})),\\
&&f_\psi(s,\hat{w},\xi)=f(\tilde{\psi}(s,\hat{w}),\xi).
\end{eqnarray*}

 We will first show that $f_\psi$ has a weak trace at $s=0$, which does not depend on the
deformation $\psi$, i.e., the way chosen to reach the boundary.
\begin{lemm}\label{weak_trace}
Let $f$ be a solution of \eqref{eq_cinetique} in
$Q_\alpha\times(-L,L)$. Then there exists
$\ftau\in\Ly({(-r_\alpha,r_\alpha)^d}\times (-L,L))$ such that
$$
\esslim_{s\to0}  f_\psi(s,\cdot,\cdot)=\ftau \ \mathrm{in} \
H^{-1}({(-r_\alpha,r_\alpha)^d} \times (-L,L)),
$$
for all $\Gamma_\alpha$-regular deformation $\psi$. Moreover,
$\ftau$ is uniquely defined.
\end{lemm}

\begin{proof}
Since $\|f_\psi(s,\cdot,\cdot)\|_\Ly\le1$, by weak compactness and
the Sobolev imbedding theorem, for every sequence $s^k\overset{k\to\infty}{\to} 0$
there exists a subsequence $k_p\overset{p\to\infty}{\to} \infty$ and a
function $g^{\tau}_\psi\in L^\infty((-r_\alpha,r_\alpha)^d\times(-L,L))$
such that
\begin{equation}\label{trace_inconnue}
    f_\psi(s^{k_p},\cdot,\cdot)
    \overset{p\to\infty}{\to} g_\psi^{\tau} \quad
    \text{in $H^{-1}\cap L^\infty_{\text{weak-$\star$}}$},
\end{equation}
for every regular deformation $\psi$. Let us show that
$g_\psi^{\tau}$ is independent of the deformation $\psi$ and the
sequence $s^{k}$ and its subsequence $s^{k_p}$. To do so, let us
first consider the entropy flux
\begin{equation}\label{def_fpsi}
    \overline{q}_{\eta}(w)=\int_{-L}^{L}a(\xi)\eta'(\xi)f(w,\xi)\,d\xi,
\end{equation}
associated with the entropy $\eta$. Multiplying \eqref{eq_cinetique}
by $\eta'(\xi)$ and integrating it with respect to $\xi$ we find
$$
\Div_y\overline{q}_{\eta}=-\int_{-L}^{L}[\eta''(\xi)m_1-\eta'(\xi)m_2](w,d\xi)
\in \mathcal{M}({(-r_\alpha,r_\alpha)}^{d+1}),
$$
where
\begin{equation}\label{two_measure}
    m_1=Sf+m,\qquad
    m_2=-\partial_\xi Sf+\delta(\xi)S.
\end{equation}

We can now use the following
theorem (cf.~Chen and Frid in \cite{C-F}):
\begin{theo}
Let $\Omega$ be an open set with regular boundary $\partial\Omega$
and $F\in\left[L^\infty(\Omega)\right]^{d+1}$ be such
that $\Div_y F$ is a bounded measure.
Then there exists $F\cdot n\in L^\infty(\partial\Omega)$ such that
for every $\partial\Omega$-regular deformation $\psi$
$$
\esslim_{s\to0} F(\psi(s,\cdot))\cdot n_s(\cdot)=
F\cdot n \quad \text{in $L^\infty_{\text{weak-$\star$}}(\partial\Omega)$},
$$
where $n_s$ is a unit outward normal field
of $\psi(\{s\}\times\partial\Omega)$.
\end{theo}

This theorem ensures the existence of a function
$\overline{q}_{\eta}^\tau\cdot n \in\Ly({(-r_\alpha,r_\alpha)^d})$, which
does not depend on $\psi$, such that
\begin{equation}\label{trace_entropy}
    \overline{q}_{\eta}(\tilde{\psi}(s,\cdot))\cdot n_{s}(\cdot)
    \overset{s\to0}{\longrightarrow}
    \overline{q}_{\eta}^\tau\cdot n \quad
    \text{in $\mathcal{D}'({(-r_\alpha,r_\alpha)^d})$,}
\end{equation}
for every regular deformation $\psi$. The function $n_s$ converges
strongly to $n$, i.e., the unit outward normal to $Q_\alpha$. The
convergence takes place in $L^1({(-r_\alpha,r_\alpha)^d})$. So,
using \eqref{def_fpsi} and \eqref{trace_inconnue},
\eqref{trace_entropy}, we obtain
$$
\int\limits_{(-r_\alpha,r_\alpha)^d}\int_{-L}^{L}
\varphi(\wh)\eta'(\xi)a(\xi)\cdot n(\wh)g_\psi^{\tau}(\wh,\xi)\,d\xi\,d\wh
=\int\limits_{(-r_\alpha,r_\alpha)^d}
\overline{q}_{\eta}^\tau\cdot
n(\wh)\varphi(\wh)\,d\wh,
$$
for every test functions
$\varphi\in\mathcal{D}({(-r_\alpha,r_\alpha)^d})$. The right-hand
side of this equation is independent of $\psi$, the sequence $s^k$
and its subsequence $s^{k_p}$, so $g_\psi^{\tau}$ does not depend on
those quantities either thanks to (\ref{hypothesis}). The result is
obtained from the uniqueness of the limit.
\end{proof}

\subsection{Strong boundary trace}
Let us now show that entropy solutions possess a strong boundary trace.
To do so we will employ the blow-up method \cite{Vasseur} and apply the averaging lemma
to conclude that $\ftau(\wh,\cdot)$ is a \chif\ for almost every $(\wh,\xi)\in
{(-r_\alpha,r_\alpha)^d}\times (-L,L)$.
To this end, we shall rely on the following
lemma, which is a straightforward
consequence of  Lemma \ref{elementaire}.

\begin{lemm}\label{trace_strong}
The function $f^\tau$ is a $\chi$-function if and only if
$$
\esslim_{s\to0}
f_\psi(s,\cdot,\cdot)=f^\tau
\quad \text{in $L^1((-r_\alpha,r_\alpha)^d)$},
$$
for any deformation $\psi$.
\end{lemm}

Let fix a specific deformation on $Q_\alpha$, namely
\begin{equation}\label{eqpsi0}
    \tilde{\psi}_0(s,\wh)=(s+\gamma_\alpha(\wh),\wh).
\end{equation}
We use the notation
\begin{equation*}
    \tf(s,\wh,\xi)=f_{\tilde{\psi}_0}(s,\wh,\xi)=f(\tilde{\psi}_0(s,\wh),\xi),
\end{equation*}
when we work with the deformation \eqref{eqpsi0}. Indeed, it is
enough to show strong trace of $f_\psi$ for the specific deformation
thanks to Lemma \ref{weak_trace}.
Notice that $\tilde{\psi}_0(s,\wh)\in Q_\alpha$ if and
only if $\wh\in{(-r_\alpha,r_\alpha)^d}$ and $0<s<r_\alpha$. From
\eqref{eq_cinetique} we find that  $\tf$ is a solution of
\begin{equation}\label{eq_cinetique_tilde}
    \tao(\wh,\xi)\partial_{\yo}\tf+\ah(\xi)\partial_{\wh}\tf
    =\partial_\xi \tm_1+\tm_2,
\end{equation}
where $\tm_i(s,\wh,\xi)=m_i(\tilde{\psi}_0(s,\wh),\xi)$ with $m_i$ defined
in \eqref{two_measure}, $i=1,2$ and
$\tao(\wh,\xi)=\lambda(\wh)a(\xi)\cdot n(\wh)$.

Before introducing the notion of rescaled solution, let us state two
lemmas (cf.~Vasseur \cite{Vasseur}).
\begin{lemm}\label{mesure_vanish}
There exists a sequence $\delta_k$ which converges to $0$ and a set
$\mathcal{E}\subset{(-r_\alpha,r_\alpha)^d}$ with
$\mathcal{L}({(-r_\alpha,r_\alpha)^d}\setminus\mathcal{E})=0$ such
that for every $\wh\in\mathcal{E}$ and every $R>0$
\begin{equation*}
    \lim_{k\to\infty}\frac{1}{\delta_n^d}\tilde{m_i}\Bigl((0,R\delta_k)\times(\wh+
    (-R\delta_k, R\delta_k)^d)\times {(-L,L)}\Bigr)=0, \quad i=1,2.
\end{equation*}
\end{lemm}

\begin{lemm}\label{lemme2}
There exists a subsequence, still denoted by $\delta_k$, and a subset
$\mathcal{E}'$ of ${(-r_\alpha,r_\alpha)^d}$ with
$\mathcal{E}'\subset\mathcal{E}$, $\mathcal{L}({(-r_\alpha,r_\alpha)^d}
\setminus\mathcal{E}')=0$, such that
for every $\wh\in\mathcal{E}'$ and every $R>0$ there holds
\begin{eqnarray*}
&&\qquad\lim_{\delta_k\to0}\int_{-L}^L\int_{(-R,R)^d}|\ftau(\wh,\xi)-\ftau(\wh+
\delta_k\byh,\xi)|\,d\byh\,d\xi=0,\\
&&\qquad\lim_{\delta_k\to0}\int_{-L}^L\int_{(-R,R)^d}|\tao(\wh,\xi)-\tao(\wh+\delta_k\byh,\xi)|\,d\byh\,d\xi=0.
\end{eqnarray*}
\end{lemm}

Let us now introduce the localization method \cite{Vasseur}. We use the notation
$$
{Q^\delta_\alpha}={(0,r_\alpha/\delta)}
\times{(-r_\alpha/\delta,r_\alpha/\delta)^d}.
$$
The goal is to show that for every $\wh\in\mathcal{E}'$,
$\ftau(\wh,\cdot)$ is a \chif. From now on we fix such a $\wh\in\mathcal{E}'$.
Then we rescale the $\tf$ function by
introducing a new function $\tfe$, which depends on new variables
$(\byo,\byh)\in{Q^\delta_\alpha}$, defined by
\begin{equation*}
    \tfe(\byo,\byh,\xi)=\tf(\delta\byo,\wh+\delta\byh,\xi).
\end{equation*}
This function depends obviously on $\wh$ but since it is fixed throughout
this section, we skip it in the notation. The function $\tfe$ is
still a \chif\ and we notice that
\begin{equation*}
    \tfe(0,\byh,\xi)=\ftau(\wh+\delta\byh,\xi).
\end{equation*}
Hence we gain knowledge about $\ftau(\wh,\cdot)$ by studying the
limit of $\tfe$ when $\delta\to0$. We define
$$
\tao_\delta(\byh,\xi)=\tao(\wh+\delta\byh,\xi).
$$

In view of
\eqref{eq_cinetique_tilde},
\begin{equation}\label{eq_cinetique_locale}
   \tao_\delta(\byh,\xi) \partial_{\byo}\tfe+\ah(\xi)\partial_{\byh}\tfe=\partial_\xi\tme^1+\tme^2,
\end{equation}
where ${\tme}^i$ is the nonnegative measure defined for every real
numbers $R^j_1<R^j_2$, $L_1<L_2$ by
\begin{equation*}
    \tme^i\Bigl(\mathop{\Pi}\limits_{0\le j\le d}
    [R^j_1,R^j_2]\times[L_1,L_2]\Bigr)
    =\frac{1}{\delta^d}\tm_i\Bigl(\mathop{\Pi}\limits_{0\le j\le d}
    [y_j+\delta R^j_1, y_j+\delta R^j_2]\times[L_1,L_2]\Bigr),
\end{equation*}
for $i=1,2$.

We now pass to the limit when $\delta$ goes to 0 in
the rescaled equation. To this end, we shall need to prove strong convergence via
an application of an averaging lemma taken from Perthame and Souganidis \cite{P-S}.

\begin{theo}\label{lemmas}
Let $N$ be an integer, $f_n$  bounded in  $L^\infty(\R^{N+1})$ and
$\{h_{n}^1,h_{n}^2\}$ be relatively compact in $\left[L^p
(\R^{N+1})\right]^{2N}$ with  $1<p<+\infty$ solutions of the
transport equation:
\[
a(\xi)\cdot\nabla_{y}f_k=\partial_\xi(\nabla_{y}\cdot h^1_{k})+
\nabla_{y}\cdot h^2_{k},
\]
where $a\in\left[C^2(\R)\right]^{N}$ verifies the non-degeneracy
condition \eqref{hypothesis}. Let $\phi\in\mathcal{D}(\R)$, then the
average $u^\phi_{k}(w)=\int_{\R}\phi(\xi)f_{k}(w,\xi)\,d\xi$ is
relatively compact in $L^p(\R^{N})$.
\end{theo}

\begin{lemm}\label{lem_limite}
There exist a sequence $\delta_k\to 0$ and a \chif\
$\tfi\in\Ly(\R^+\times\R\times (-L,L))$ such that $\tfen$ converges
strongly to $\tfi$ in $L^1_{\mathrm{loc}}(\R^+\times\R\times
(-L,L))$ and
\begin{equation}\label{eq_cinetique_limite}
   \tao(\wh,\xi) \partial_{\byo}\tfi+\ah(\xi)\cdot\partial_{\bwh}\tfi=0.
\end{equation}
\end{lemm}

\begin{proof}
We consider the sequence $\delta_n$ of Lemma~\ref{lemme2}. By weak
compactness, there exists a function $\tfi \in \Ly(\R^+\times
\R^d\times (-L,L))$ such that, up to extraction of a subsequence, $\tfen$ converges to
$\tfi$ in $L^\infty_{\text{weak-$\star$}}$. Thanks to Lemma~\ref{mesure_vanish},
$\tmen^i$ converges to 0 in the sense of measures. So
passing to the limit in \eqref{eq_cinetique_locale} gives
\eqref{eq_cinetique_limite}.

First, we localize in $(\bw,\xi)$. For any $R>0$ big enough, we
consider $\Phi_1,\Phi_2$ with values in $[0,1]$ such that
$\Phi_1\in\mathcal{D}(\R^+\times\R^d)$, $\Phi_2\in\mathcal{D}(\R)$,
and $\mathrm{Supp}(\Phi_1)\subset(1/(2R),2R)\times{(-2R,2R)^d}$,
$\mathrm{Supp}(\Phi_2)\subset (-2L,2L)$. Moreover, $\Phi_1(\bw)=1$
for $\bw\in(1/R,R)\times{(-R,R)^d}$ and $\Phi_2(\xi)=1$ for $\xi\in
(-L,L)$. Hence for $\delta<r_\alpha/(2R)$, we can define on
$\R\times\R^d\times\R$ the function
$$
\tilde{f}^{R}_\delta=\Phi_1\Phi_2\tilde{f}_\delta,
$$
(where $\tilde{f}^{R}_\delta=0$ if $\tfe$ is not defined). On
$(1/R,R)\times{(-R,R)}\times (-L,L)$ we have
$\tilde{f}^{R}_\delta=\tfe$. So, if we denote by
$a_{\wh}(\xi)=(\tao(\wh,\xi),\ah(\xi))$ (which depends only on $\xi$
since $\wh$ is fixed), from \eqref{eq_cinetique_locale} we get
\begin{align*}
    a_{\wh}(\xi)\cdot\nabla_{\bwh}\tilde{f}_\delta^{R} &
    =\partial_\xi(\Phi_1\Phi_2\tme^1)-\Phi_1\Phi_2'\tme^1
    +a_{\wh}(\xi)\cdot\nabla_{\bwh}\Phi_1\Phi_2\tfe^R
    +\Phi_1\Phi_2\tilde{m}^2_\delta\\
    &+\partial_s[(\tao(\wh,\xi)-\tao_\delta(\byh,\xi))\tilde{f}_\delta^R]\\
    & =\partial_\xi\mu_{1,\delta}+\mu_{2,\delta}+\partial_s[(\tao(\wh,\xi)-\tao_\delta(\byh,\xi))\tilde{f}_\delta^R],
\end{align*}
where $\mu_{1,\delta_k}$
and $\mu_{2,\delta_k}$ are measures uniformly bounded with respect
to $k$. In view of Lemma \ref{lemme2} we can see that
$\tao(\wh,\xi)-\tao_\delta(\byh,\xi)$ converges to $0$ in
$L^1_{\mathrm{loc}}(\R^d\times(-L,L))$. So it converges to 0 in
$L^p_{\mathrm{loc}}$ for every $1\le p<\infty$ since these
functions are bounded in $L^\infty$. Since the measures are
compactly imbedded in $W^{-1,p}$ for $1\le p<\frac{d+2}{d+1}$, we
can apply Theorem~\ref{lemmas} with $N=d+1$,
$f_k=\tilde{f}^{R}_{\delta_k}$, $\phi(\xi)=\Phi_2(\xi)$, and
$a(\xi)=a_{\wh}(\xi)$. It follows that $\int
\tilde{f}^{R}_\delta\Phi_2(\xi)\,d\xi$ is compact in $L^p$ for
$1\le p<\frac{d+2}{d+1}$. And so by uniqueness of the limit, $\int
\tfen(\cdot,\xi)\,d\xi$ converges strongly to
$\int\tfi(\cdot,\xi)\,d\xi$ in $L^1_{\mathrm{loc}}(\R^{2})$.
Lemma~\ref{elementaire} ensures us that $\tfen$ converges strongly
to $\tfi$ in $L^1_{\mathrm{loc}} (\R^{d+1}\times (-L,L))$ and
moreover that $\tfi$ is a \chif.
\end{proof}

We now turn to the characterization of the limit function $\tilde{f}_\infty$.
\begin{lemm}[\cite{Vasseur}]\label{trace_strong1}
For every $\wh\in\mathcal{E}'$,
$\tilde{f}_\infty(\underline{w},\xi)=f^\tau(\wh,\xi)$ for almost
every $(\underline{w},\xi)\in\R^{d+1}\times(-L,L)$, and the function
$f^\tau(\wh,\cdot)$ is a $\chi$-function.
\end{lemm}

Thus, from Propositions \ref{trace_strong}
and \ref{trace_strong1}, we can prove Theorem \ref{the_theorem1}.

\begin{proof}[Proof of Theorem \ref{the_theorem1}]
For every $\alpha$ and every deformation $\psi$, we have
\begin{equation*}
    \esslim_{s\to0}\int\limits_{(-r_\alpha,r_\alpha)^d}
    \int_{-L}^L |f_\psi(s,\wh,\xi)-f^\tau(\wh,\xi)|d\xi d\wh=0.
\end{equation*}
We define $u^\tau$ by
$$
u^\tau(\hat{z})=\int_{-L}^L f^\tau(\wh,\xi)d\xi,
\qquad \text{if $(\gamma_\alpha(\wh),\wh)=\Lambda_\alpha(\hat{z})$.}
$$
For every compact subset $K$ of $(0,T)\times\partial\Omega$, there exists a
finite set $I_0$ such that $K\subset\bigcup_{\alpha\in I_0}$ and
$$
\int_K|u(\psi(s,\hat{z}))-u^\tau(\hat{z})|d\sigma(\hat{z})\le\sum_{\alpha\in
I_0}\int_{\Gamma_\alpha}|u(\psi(s,\hat{z}))-u^\tau(\hat{z})|d\sigma(\hat{z}),
$$
which converges to $0$ as $s$ tends to $0$.
This concludes the proof of Theorem \ref{the_theorem1}.
\end{proof}

\section{Proof of Theorem \ref{the_theorem2}}\label{sec:proof-existence-uniqueness}

\subsection{Existence proof}\label{subsec:proof-existence}
In this section we will show the existence of an entropy solution for
the initial-boundary value problem \eqref{eq_moment}, \eqref {eq:initial}, and \eqref{eq:boundary},
with the boundary condition \eqref{eq:boundary}
interpreted in the sense of \eqref{DL_boundary}.

Let $\{S^\eps\}_{\eps>0}$ be a a sequence of smooth functions
converging in $L^1_{\mathrm{loc}}$ to $S$ with respect to variables
$(t,x)$, for example obtained by mollifying the function $S$, and
consider smooth solutions to the uniformly parabolic equation
\begin{equation}\label{diffusion_eq}
    \partial_t u^\eps+\Div_x A(u^\eps)
    =S^\eps(t,x,u^\eps)+\eps\Delta_xu^\eps,
\end{equation}
with initial and boundary data
\begin{equation}\label{eq:initial-boundary-viscous}
    u^\eps(0,\cdot)=u_0 \qquad
    u^\eps |_{\Gamma}=u_b.
\end{equation}
For the sake of simplicity in this proof, we will
assume that the data $u_0,u_b$ are smooth functions. Then, for each
$\eps>0$, the existence of a unique smooth solution of the
initial-boundary \eqref{diffusion_eq},
\eqref{eq:initial-boundary-viscous} value problem is a standard
result.

By the maximum principle,
\begin{equation}\label{max_principle}
    |u^\eps(t,x)|\le
    \norm{u_0}_{L^\infty}+\norm{S}_{L^\infty}T.
\end{equation}

For any convex entropy function $\eta$ and corresponding
entropy flux function $q$ with $q'=\eta' A'$, multiplying \eqref{diffusion_eq}
$\eta'(u^\eps)$ yields
\begin{equation}\label{compact_approximation}
    \partial_t \eta(u^\eps)
    +\Div_xq(u^\eps)-\eta'(u^\eps)S^\eps(t,x,u^\eps)=
    \eps\Delta_x\eta(u^\eps)-\eps\eta''(u^\eps)
    |\nabla_xu^\eps|^2.
\end{equation}
For any function $\varphi\in C^\infty_\mathrm{c}(Q)$, it follows
from \eqref{compact_approximation} that
\begin{equation}\label{diffusion_inequality}
    \begin{split}
        &\int_{Q} \eta(u^\eps)\partial_t\varphi +q(u^\eps)\cdot\nabla_x\varphi \, dtdx
        \\ &
        =\int_{Q}\eps \eta'(u^\eps)\nabla_x u^\eps \cdot \nabla_x \varphi \, dtdx
        +\int_Q \eta''(u^\eps)\eps|\nabla_xu^\eps|^2\varphi \, dtdx
        \\ & \qquad
        -\int_{Q}S^\eps(t,x,u^\eps)u^\eps\varphi \, dtdx.
    \end{split}
\end{equation}
Let $K$ be an arbitrary compact subset of $Q$ and
choose in \eqref{diffusion_inequality} a
function $\varphi\in C^\infty_\mathrm{c}(Q)$ satisfying
$$
\varphi|_K=1 ,\qquad 0\le\varphi \le1.
$$
It follows that
\begin{equation*}
    \int_Q |S^\eps(t,x,u^\eps)u^\eps\varphi| \, dtdx
    \le C(T,\varphi,\|u_0\|_{L^\infty})\norm{S}_{L^\infty},
\end{equation*}
thanks to \eqref{max_principle}. Consequently,
\begin{equation}\label{compact_eq_inequality}
    \int_Q \eps \abs{\nabla_x u^\eps}^2 \, dtdx\le C
\end{equation}
and hence we obtain that $\partial_t \eta(u^\eps)+\Div_x q(u^\eps)$
is compact in $H^{-1}_{\mathrm{loc}}(Q)$.
We can now apply, for example, Tartar's compensated compactness
method \cite{Tartar:1983ul} to conclude the existence of subsequence, still
labeled $u^\eps$, converging to a limit $u$ a.e.~and
in $\Lenloc$ such that the interior entropy inequality holds:
\begin{equation*}
    \int_Q \eta(u)\partial_t\phi
    +q(u)\cdot\nabla_x\phi+\eta(u)S(t,x,u)\phi \, dtdx\geq 0,
    \quad \forall \phi\in C^\infty_{c}(Q), \phi\ge0.
\end{equation*}
It remains to prove that the limit $u$ satisfies
the Dubois and Le Floch's boundary condition \eqref{DL_boundary}.

\begin{lemm}\label{BL_lemma}
Let $u$ be the limit function constructed above.
Then, for any convex entropy-entropy flux pair $(\eta,q)$,
$$
\Bigl[q(u^\tau)-q(u_b)-\eta'(u_b)(A(u^\tau)-A(u_b))\Bigr]
\cdot\hat{n}\geq0
$$
where $B^\tau$ is the trace of $B$ on
$(0,T)\times\partial\Omega$ and  $\hat{n}$ is the unit
outward normal to $\partial\Omega$.
\end{lemm}

\begin{proof}
We need a family of boundary layer
functions $\{\zeta_\delta\} \in C^\infty(\Omega;[0,1])$ verifying
$$
\zeta_\delta|_{\Omega_\delta}=0, \quad
\zeta_\delta|_{\partial\Omega}=1, \quad
\mathrm{and}\quad
|\nabla\zeta|\le\frac{c}{\delta^d},
$$
where $\Omega_\delta=\{x\in\Omega|\mathrm{diam}(x,\partial\Omega)>\delta\}$
and $c$ is a constant independent of $\delta$.

Multiplying  \eqref{compact_approximation} by $\theta(t)\zeta_\delta(x)$
with $\theta\in C^\infty_c(0,T)$, $\theta\ge 0$, we obtain
$E_1=E_2$, where the terms $E_1,E_2$ are defined and analyzed below.

Integration by parts yields
\begin{equation*}
    \begin{split}
        E_1 & :=\int_Q \Bigl(\partial_t \eta(u^\eps)+\Div_x q(u^\eps)
        -\eta'(u^\eps)S^\eps(t,x,u^\eps)\Bigr)\theta(t)\zeta_\delta(x)\, dtdx
        \\ & = - \int_Q  \eta(u^\eps)\theta'(t)\zeta_\delta(x)
        +q(u^\eps)\cdot\nabla_x \zeta_\delta(x)\theta(t)
        \\ & \qquad\qquad + \eta'(u^\eps) S^\eps(t,x,u^\eps)
        \theta(t)\zeta_\delta(x)\, dtdx
        +\int_{(0,T)\times \partial\Omega} q(u_b)\cdot \hat{n}\, \theta(t)\, dtdx
        \\ & \overset{\eps\to 0}{\to}
        -\int_Q  \eta(u)\theta'(t)\zeta_\delta(x)
        +q(u)\cdot \nabla_x \zeta_\delta(x)\theta(t)
        +\eta'(u)S(t,x,u))
        \theta(t)\zeta_\delta(x) \, dtdx
        \\ & \qquad\qquad+\int_{(0,T)\times \partial\Omega}q(u_b)
        \cdot \hat{n} \, \theta(t)\, dtd\sigma.
    \end{split}
\end{equation*}
Observe that
$$
\int_Q q(u)\cdot\nabla_x \zeta_\delta(x)\theta(t)\, dtdx
\overset{\delta\to 0}{\to}
\int_0^T\int_{\partial\Omega}q(u^\tau)\cdot \hat{n}\, \theta(t)\, dtd\sigma
$$
and
$$
\int_Q\eta'(u)S(t,x,u)\theta(t)\zeta_\delta(x)\, dtdx
\overset{\delta\to 0}{\to} 0,\quad
\int_Q\eta(u)\theta'(t)\zeta_\delta(x)\,dtdx\overset{\delta\to0}{\to}
0.
$$
As a result,
\begin{equation*}
    \lim_{\delta\to 0}\lim_{\eps\to 0} E_1=
    \int_{(0,T)\times \partial\Omega}
    (q(u_b)-q(u^\tau))\cdot \hat{n})\theta(t)\, dtd\sigma.
\end{equation*}

Next,

\begin{equation*}
    \begin{split}
        E_2 & := \eps\int_Q\eta'(u^\eps)\Delta_x u^\eps\theta(t)\zeta_\delta(x)\, dtdx
        \\ & =\eps\int_Q \Bigl(\Div_x (\eta'(u^\eps)\nabla_x u^\eps)-\eta''(u^\eps)
        |\nabla_x u^\eps|^2\Bigr)\theta(t)\zeta_\delta(x)\, dtdx\\
        & \le\eps \int_{(0,t)\times\partial\Omega}\eta'(u_b)\nabla_x u^\eps\cdot \hat{n}\, \theta(t)\, dtd\sigma
        -\eps\int_Q \eta'(u^\eps)\nabla_x u^\eps\cdot\nabla_x \zeta_\delta(x)\theta(t)\, dtdx
        \\ & =: E_{2,1}-E_{2,2}.
    \end{split}
\end{equation*}

Clearly, thanks to \eqref{compact_eq_inequality},
$\displaystyle \lim_{\eps\to 0} \abs{E_{2,2}}=0$.

To analyze $E_{2,1}$, we repeat the above argument with
$\eta=\text{Id}$ to obtain the equation
$$
\lim_{\delta\to 0}\lim_{\eps \to 0}
\Biggl[\eps\int_{(0,T)\times \partial\Omega}\nabla_x u^\eps\cdot
\hat{n} \, \theta(t)\,dtd\sigma\Biggr]
=\int_{(0,T)\times \partial\Omega}(A(u^\tau)-A(u_b))
\cdot \hat{n}\,\theta(t)\, dtd\sigma,
$$
and consequently
$$
\lim_{\eps\to 0}\lim_{\delta \to 0} E_{2,1}=
\int_{(0,T)\times \partial\Omega}
\eta'(u_b)(A(u^\tau)-A(u_b))\cdot \hat{n}\,\theta(t)\, dtd\sigma;
$$
hence the limit $u$ obeys the inequality
$$
\int_0^T\int_{\partial\Omega}\Bigl[ q(u^\tau)-q(u_b)
-\eta'(u_b)(A(u^\tau)-A(u_b))\Bigr]\cdot \hat{n}\, \theta(t)\,d\sigma dt\geq0.
$$
By the arbitrariness of $\theta$, the proof is complete.
\end{proof}

\subsection{Uniqueness proof}
In this section we prove the uniqueness part of Theorem
\ref{the_theorem2}, adapting the approach of Perthame \cite{Perthame,Perthame:2002qy}.
In what follows, we let $u,v$ denote two entropy solutions of the
conservation law \eqref{eq_moment} with initial data $u_0,v_0\in L^\infty$, respectively, and
boundary data $u_b$, with the boundary condition \eqref{eq:boundary}
interpreted in the sense of \eqref{DL_boundary}.
We start by rewriting the Dubois and LeFloch boundary condition
\eqref{DL_boundary} in a kinetic form due to Kwon \cite{Kwon}.

\begin{lemm}\label{kinetic_form}
The following two statements are equivalent:

\noindent
1. For every convex entropy-entropy flux pair $(\eta,q)$,
$$
\Bigl[q(u^\tau)-q(u_b)-\eta'(u_b)(A(u^\tau)-A(u_b))\Bigr]\cdot\hat{n}\geq0
\quad \text{on $\Gamma$}.
$$
2. There exists $\mu\in \mathcal{M}^+(\Gamma\times(-L,L))$
such that
$$
A'(\xi)\cdot
\hat{n}(\xh)\bigl[f^\tau(\xh,\xi)-\chi(\xi;u_b(\xh))\bigr]
-\delta_{(\xi=u_b(\xh))}(A(u^\tau)-A(u_b))\cdot\hat{n}=-\partial_\xi\mu(\xh,\xi),
$$
for every $(\xh,\xi)\in\Gamma\times(-L,L)$.
\end{lemm}

Associated with the entropy solutions $u$ and $v$ we introduce the corresponding
$\chi$-functions $f$ and $g$ defined by $f(t,x,\xi)=\chi(\xi;u(t,x))$ and
$g(t,x,\xi)=\chi(\xi;v(t,x))$, respectively. In view of Theorem \ref{theo_LPT}, there exist
$m^1,m^2\in\mathcal{M}^+(Q\times(-L,L))$ such that
\begin{equation}\label{kinetic_equation}
    \begin{split}
        & \partial_t f+A'(\xi)\cdot\nabla_x f
        -S(t,x,\xi)(\partial_\xi f-\delta(\xi)) =\partial_\xi m^1,\\
        & \partial_t g+A'(\xi)\cdot\nabla_x g
        -S(t,x,\xi)(\partial_\xi g-\delta(\xi))=\partial_\xi m^2.
    \end{split}
\end{equation}

The goal is to show the following inequality for a.e.~$t\in(0,T)$:
\begin{equation}\label{regular_equation}
    \begin{split}
        &\frac{d}{dt}\int_{\Omega}\int_{-L}^L|f(t,x,\xi)-g(t,x,\xi)|^2\, d\xi dx\\
        & \quad+\int_{\partial\Omega}
        \int_{-L}^L A'(\xi)\cdot \hat{n}(x)|f^\tau(t,x,\xi)-g^\tau(t,x,\xi)|^2\,d\xi d\sigma\\
        &
        \quad\quad\le C\int_{\Omega}|S(t,x,u(t,x))-S(t,x,v(t,x))|\,dx,
    \end{split}
\end{equation}
where $d\sigma$ denotes the volume element of $\partial\Omega$ and
some constants $C>0$.

To this end, we need to regularize $f$ and $g$ with
respect to the $t,x$ variables. Set
$\epsilon=(\epsilon_1,\epsilon_2)$ and
define $\phi_{\epsilon}$ by
$$
\phi_{\epsilon}(t,x)=\frac{1}{\epsilon_1}\phi_1\left(\frac{t}{\epsilon_1}\right)
\frac{1}{\epsilon_2^d}\phi_2\left(\frac{x}{\epsilon_2}\right),
$$
where $\phi_1\in C_c^\infty(\R)$, $\phi_2\in C_c^\infty(\R^d)$ verify
$\phi_j\geq0$, $\int\phi_j=1$ for $j=1,2$, and $\mathrm{supp}(\phi_1)\subset(0,1)$. We shall
employ the following notations:
\begin{align*}
    f_\epsilon(t,x,\xi)=f(\cdot,\cdot,\xi)\underset{(t,x)}{\star}\phi_\epsilon(t,x),&\quad
    g_\epsilon(t,x,\xi)=g(\cdot,\cdot,\xi)\underset{(t,x)}{\star}\phi_\epsilon(t,x),\\
    m_{\epsilon}^1(t,x,\xi)=m^1(\cdot,\cdot,\xi)\underset{(t,x)}{\star}\phi_\epsilon(t,x),&\quad
    m_{\epsilon}^2(t,x,\xi)=m^2(\cdot,\cdot,\xi)\underset{(t,x)}{\star}\phi_\epsilon(t,x),
\end{align*}
where $\star$ means convolution with respect to the indicated variables and
the mappings $f,g,m_1,m_2$ are extended to $\R^{d+1}$ by letting them take
the value zero on $\R^{d+1}\setminus Q$.

The proof of the following lemma can be
found in Perthame \cite{Perthame,Perthame:2002qy}.

\begin{lemm}\label{defect_measure}
Let $m^1$ and $m^2$ be non-negative measures given in the Theorem
\ref{theo_LPT}. Then, the following holds
\begin{equation*}
    \lim_{\epsilon\rightarrow0}\int_{-L}^L
    m_\epsilon^1(\cdot,\cdot,\xi)\delta_{(\xi=u)}\ast\phi_\epsilon
    +m_\epsilon^2(\cdot,\cdot,\xi)\delta_{(\xi=v)}\ast\phi_\epsilon d\xi=0
    \quad \text{in $\mathcal{D}'(Q)$}.
\end{equation*}
\end{lemm}

Let us continue with the proof of \eqref{regular_equation}.
Fix a $\partial\Omega$-regular deformation $\hat{\psi}$, and let $\Omega_s$ denote
the open subset of $\Omega$ whose boundary is
$\partial\Omega_s= \hat{\psi}(\{s\}\times\partial\Omega)$. Taking the convolution
of each of the two kinetic equations in \eqref{kinetic_equation} and then
subtracting the resulting equations we obtain an equation that is
multiplied by $f_\epsilon-g_\epsilon$. The final outcome reads
\begin{equation}\label{regular_equation2}
    \begin{split}
        & \int_{\Omega_s}\int_{-L}^L\partial_t|f_\epsilon(t,x,\xi)-g_\epsilon(t,x,\xi)|^2
        +A'(\xi)\cdot\nabla_x|f_\epsilon(t,x,\xi)-g_\epsilon(t,x,\xi)|^2\,d\xi d\sigma_s\\
        & \qquad-\int_{\Omega_s}\int_{-L}^L [S(t,x,\xi)(\partial_\xi(f-g))]
        \underset{(t,x)}{\star}\phi_\epsilon(t,x)(f_\epsilon(t,x,\xi)-g_\epsilon(t,x,\xi))\,d\xi d\sigma_s \\
        & =2\int_{\Omega_s}\int_{-L}^L\partial_\xi(m_\epsilon^1(t,x,\xi)-m_\epsilon^2(t,x,\xi))
        (f_\epsilon(t,x,\xi)-g_\epsilon(t,x,\xi))\, d\xi d\sigma_s,
    \end{split}
\end{equation}
for a.e.~$s>0$, where $d\sigma_s$ denotes the volume element of
$\partial\Omega_s$.

In view of Lemma \ref{defect_measure}, observe that for a.e.~$s>0$ we have
\begin{equation*}
    \begin{split}
        &\lim_{\epsilon\rightarrow0}\int_{\Omega_s}
        \int_{-L}^L\partial_\xi(m_\epsilon^1(\cdot,\cdot,\xi)-m_\epsilon^2(\cdot,\cdot,\xi))
        (f_\epsilon(\cdot,\cdot,\xi)-g_\epsilon(\cdot,\cdot,\xi))\,d\xi d\sigma_s\\
        & =-\lim_{\epsilon\rightarrow0}\int_{\Omega_s}\int_{-L}^L
        (m_\epsilon^1(\cdot,\cdot,\xi)-m_\epsilon^2(\cdot,\cdot,\xi))
        \partial_\xi(f_\epsilon(\cdot,\cdot,\xi)-g_\epsilon(\cdot,\cdot,\xi))\,d\xi d\sigma_s\\
        & =-\lim_{\epsilon\rightarrow0}\int_{\Omega_s}\int_{-L}^Lm_\epsilon^1(\cdot,\cdot,\xi)
        \delta_{(\xi=v)}\underset{(t,x)}{\star}\phi_\epsilon
        +m_\epsilon^2(\cdot,\cdot,\xi)\delta_{(\xi=u)}\underset{(t,x)}{\star}\phi_\epsilon\,d\xi d\sigma_s
        \le 0.
    \end{split}
\end{equation*}

Next, observe that
\begin{align*}
    & \limsup_{\epsilon\rightarrow0}|\int_{\Omega_s}\int_{-L}^L
    [S(t,x,\xi)(\partial_\xi(f-g))]\underset{(t,x)}{\star}
    \phi_\epsilon(t,x)(f_\epsilon(t,x,\xi)-g_\epsilon(t,x,\xi))\, d\xi dx|\\
    & \qquad \le 2\int_{\Omega_s}|S(t,x,u)-S(t,x,v)|\,dx
    \le  2C\int_{\Omega_s}|u-v|\,dx, \quad
    \text{for a.e.~$s>0$,}
\end{align*}
where we have used condition \eqref{eq:source-ass} to derive the
last inequality. Indeed, using $|f|\leq 1$ and $|g|\leq 1$, we 
obtain $|f_\epsilon-g_\epsilon|\leq2$ and we can easily check that for 
$a.e.~(t,x)\in (0,T)\times\Omega$,
$$
\int_{-L}^L
[S(t,x,\xi)(\partial_\xi(f-g))]\underset{(t,x)}{\star}
\phi_\epsilon(t,x) d\xi\overset{\epsilon\to 0}{\longrightarrow} 
S(t,x,v)-S(t,x,u),
$$
thanks to $\partial_\xi(f-g)=\delta(\xi=v)-\delta(\xi=u)$.

 Let us now apply the divergence theorem in
\eqref{regular_equation2} and subsequently take the limits
$\epsilon\to0$ and $s\to0$. Applying Theorem \ref{the_theorem1} and
the observations above, we obtain the following inequality for
a.e.~$t\in (0,T)$:
\begin{equation}\label{regular_equation1}
    \begin{split}
        & \int_{\Omega}\int_{-L}^L\partial_t|f(t,x,\xi)-g(t,x,\xi)|^2\, d\xi dx\\
        & \quad+\int_{\partial\Omega}\int_{-L}^L A'(\xi)\cdot \hat{n}(x)\,
        |f^\tau(t,x,\xi)-g^\tau(t,x,\xi)|^2\, d\xi d\sigma(x)\\
        & \quad\quad \le2\int_{\Omega}|S(t,x,u)-S(t,x,v)|\,dx.
    \end{split}
\end{equation}

Next, we show that the ``boundary'' part of
\eqref{regular_equation1} is non-negative. According to Lemma
\ref{kinetic_form}, there exist two measures $\mu_f, \mu_g\in
\mathcal{M}^+(\Gamma\times(-L,L))$ corresponding to $f$ and $g$,
respectively, verifying
\begin{equation}\label{kinetic_measure}
    \begin{split}
        & A'(\xi)\cdot \hat{n}(\xh)\bigl[f^\tau(\xh,\xi)-\chi(\xi;u_b(\xh))\bigr]
        -\delta_{(\xi=u_b(\xh))}(A(u^\tau(\xh))-A(u_b(\xh)))\cdot\hat{n}\\
        & \qquad =-\partial_\xi \mu_f(\xh,\xi), \\
        & A'(\xi)\cdot \hat{n}(\xh)\bigl[g^\tau(\xh,\xi)-\chi(\xi;u_b(\xh))\bigr]
        -\delta_{(\xi=u_b(\xh))}(A(v^\tau(\xh))-A(u_b(\xh)))\cdot\hat{n}\\
        & \qquad =-\partial_\xi \mu_g(\xh,\xi),
    \end{split}
\end{equation}
for $(\xh,\xi)\in\Gamma\times(-L,L)$.

For later use, observe that
\begin{equation}\label{sign_kinetic}
    \begin{split}
        &A'\cdot \hat{n}\,  |f^\tau-g^\tau|^2\\
        & \quad =A'\cdot \hat{n}\, (f^\tau-\chi(\xi;u_b))\sgn(\xi-u_b)
        -2A'\cdot \hat{n}\, (f^\tau-\chi(\xi;u_b))(g^\tau-\chi(\xi;u_b))\\
        & \qquad\quad +A'\cdot \hat{n}\, (g^\tau-\chi(\xi;u_b))\sgn(\xi-u_b) \\
        & \quad =A'\cdot \hat{n}\, (f^\tau-\chi(\xi;u_b))[\sgn(\xi-u_b)-g^\tau+\chi(\xi;u_b)]\\
        & \qquad\quad
        +A'\cdot \hat{n}\, (g^\tau-\chi(\xi;u_b))[\sgn(\xi-u_b)-f^\tau+\chi(\xi;u_b)] \\
        & \quad =: A'\cdot \hat{n}\, (f^\tau-\chi(\xi;u_b)) \alpha(\xi)
        +A'\cdot \hat{n}\, (g^\tau-\chi(\xi;u_b))\beta(\xi),
    \end{split}
\end{equation}
where $\sgn(\cdot)$ denotes the usual sign function with $\sgn(0)=0$.
Combining \eqref{kinetic_measure} and \eqref{sign_kinetic}, along with integration by parts, gives
\begin{equation}\label{last_equality}
    \begin{split}
        & \int_{\partial\Omega}\int_{-L}^L A'(\xi)\cdot \hat{n}(\xh)\,
        |f^\tau(\xh,\xi)-g^\tau(\xh,\xi)|^2\, d\xi d\sigma\\
        & =\int_{\partial\Omega}\int_{-L}^L A'(\xi)\cdot \hat{n}(\xh)\,
        [f^\tau(\xh,\xi)-\chi(\xi;u_b(\xh))]\alpha(\xi)\, d\xi d\sigma\\
        & \qquad +\int_{\partial\Omega}\int_{-L}^L A'(\xi)\cdot \hat{n}(\xh)\,
        [g^\tau(\xh,\xi)-\chi(\xi;u_b(\xh))]\beta(\xi)\, d\xi d\sigma\\
        & =\int_{\partial\Omega}\left(\int_{-L}^{u_b}+\int_{u_b}^{L}\right)
        \bigl[-\partial_\xi\mu_f(\xh,\xi)\alpha(\xi)-\partial_\xi\mu_g(\xh,\xi)\beta(\xi)\bigr]\, d\xi d\sigma\\
        & =\int_{\partial\Omega}\int_{-L}^{u_b}
        \mu_f(\xh,\xi)\, \nu_g\, d\xi d\sigma- \mu_f(u_b^{-})\alpha(u_b^{-})\\
        & \qquad+\int_{\partial\Omega}\int_{u_b}^{L}\mu_f(\xh,\xi)\, \nu_g\,d\xi d\sigma
        + \mu_f(u_b^{+})\alpha(u_b^{+})\\
        & \qquad+\int_{\partial\Omega}\int_{-L}^{u_b}\mu_g(\xh,\xi)\, \nu_f\, d\xi d\sigma
        - \mu_g(u_b^{-})\beta(u_b^{-})\\
        & \qquad
        +\int_{\partial\Omega}\int_{u_b}^{L}\mu_g(\xh,\xi)\, \nu_f \, d\xi d\sigma
        + \mu_g(u_b^{+})\beta(u_b^{+}),
    \end{split}
\end{equation}
where $\nu_f,\nu_g$ are non-negative measures defined by the
relations $\partial_\xi f^\tau=\delta(\xi)-\nu_f$ and $\partial_\xi
g^\tau=\delta(\xi)-\nu_g$, respectively.
Notice that $\alpha(u_b^{+})\geq0$, $\beta(u_b^{+})\geq0$, and
$\alpha(u_b^{-})\le0$, $\beta(u_b^{-})\le0$. Thus,
\eqref{last_equality} is non-negative.

Let us now conclude the proof of Theorem \ref{the_theorem2}. Since 
the second and third terms in \eqref{regular_equation1} are
non-negative, Gronwall's inequality imply that for each fixed $s\in (0,t)$
\begin{equation*}
    \int_{\Omega}\int_{-L}^L|f(t,x,\xi)-g(t,x,\xi)|^2\,d\xi dx
    \le\exp(2CT)\int_{\Omega}\int_{-L}^L|f(s,x,\xi)-g(s,x,\xi)|^2\,d\xi dx,
\end{equation*}
where $C$ is given in \eqref{eq:source-ass}.

 Therefore, in view of Theorem \ref{the_theorem1}, we can let
$s\to 0$ to obtain
$$
\int_\Omega|u(t,x)-v(t,x)|\,dx
\le\exp(2CT)\int_\Omega|u_0(x)-v_0(x)|\,dx,
\quad \text{for a.e.~$t\in (0,T)$.}
$$
This concludes the proof of Theorem
\ref{the_theorem2}.

\section{IBVP for the Degasperis-Procesi equation}
\label{sec:DP}

The purpose of this section is to prove Theorem \ref{the_theorem3}.
The main step of the proof relates to the existence of an entropy solution.
Our existence argument is based passing to the limit
in a vanishing viscosity approximation of \eqref{eq:DPw}.

Fix a small number $\eps>0$, and let $\ue=\ue (t,x)$ be the unique
classical solution of the following mixed problem \cite{CHK:ParEll}:
\begin{equation}\label{eq:DPepsw}
    \begin{cases}
        \pt \ue+\ue \px \ue+\px\Pe= \eps\pxx\ue,&\quad (t,x)\in (0,T)\times(0,1),\\
        -\pxx\Pe+\Pe=\frac{3}{2}u_\eps^2,&\quad (t,x)\in (0,T)\times(0,1),\\
        \ue(0,x)=u_{\eps,0}(x),&\quad x\in (0,1),\\
        \ue(t,0)=g_{\eps,0}(t),\>\>\> \ue(t,1)=g_{\eps,1}(t),&\quad t\in(0,T),\\
        \px \Pe(t,0)=\psi_{\eps,0}(t),\>\>\>  \px\Pe(t,1)=\psi_{\eps,1}(t),&\quad t\in(0,T),
    \end{cases}
\end{equation}
where $u_{\eps,0},\,g_{\eps,0},\,g_{\eps,1}$ are $C^\infty$ approximations of
$u_{0},\,g_{0},\,g_{1}$, respectively, such that
\begin{equation*}
    g_{\eps,0}(0)=u_{\eps,0}(0),\qquad g_{\eps,1}(0)=u_{\eps,0}(1),
\end{equation*}
and
\begin{equation}\label{eq:boundaryPeps}
    \psi_{\eps,0}=-g_{\eps,0}'-g_{\eps,0}h_{\eps,0},\qquad
    \psi_{\eps,1}=-g_{\eps,1}'-g_{\eps,1}h_{\eps,1}.
\end{equation}
Due to  \eqref{eq:boundaryPeps} and the first equation in \eqref{eq:DPepsw}, we have that
\begin{equation}\label{eq:u_xxeps}
    \pxx \ue(t,0)=\pxx\ue(t,1)=0,\qquad t\in(0,T).
\end{equation}

For our own convenience let us convert \eqref{eq:DPepsw} into a problem with
homogeneous boundary conditions. To this end, we introduce
the following notations:
\begin{equation}\label{eq:omega}
    \begin{split}
        \ome(t,x)= xg_{\eps,1}(t)+(1-x)g_{\eps,0}(t),&
        \qquad \ve=\ue-\ome,\\
        \Ome(t,x)=\frac{x^2}{2}\psi_{\eps,1}(t)+\frac{2x-x^2}{2}
        \psi_{\eps,0}(t),&\qquad \Ve=\Pe-\Ome.
    \end{split}
\end{equation}
Thanks to
\begin{align*}
    \ome(t,0)=g_{\eps,0}(t),&\qquad \ome(t,1)=g_{\eps,1}(t),\qquad t\in(0,T),\\
    \px\Ome(t,0)=\psi_{\eps,0}(t),&\qquad\px\Ome(t,1)=\psi_{\eps,1}(t),\qquad t\in(0,T),
\end{align*}
we have that
\begin{equation}\label{eq:boundaryv}
    \ve(t,0)=\ve(t,1)=\px\Ve(t,0)=\px\Ve(t,1)=0,\qquad t\in(0,T).
    \end{equation}
Moreover, due to the definition of $\ome$ and \eqref{eq:u_xxeps}
\begin{equation}\label{eq:v_xxeps}
    \pxx\ome(t,x)=\pxxx\Ome(t,x)=\pxx\ve(t,1)=\pxx\ve(t,0)=0,
\end{equation}
for each $t\in (0,T)$ and $x\in (0,1).$

Finally, in view of \eqref{eq:DPepsw} and \eqref{eq:v_xxeps}, we obtain
\begin{align}\label{eq:DPepsweakv}
    &\pt \ve+\pt\ome+ \ue \px\ue+\px\Pe= \eps\pxx\ve,\\
    \label{eq:DPepsweakP}
    &-\pxx\Ve+\Ve=\frac{3}{2}\ue^2+\pxx\Ome-\Ome.
\end{align}

We are now ready to state and prove our key estimate.
\begin{lemm}\label{lm:energyestimate}
For each $t\in(0,T)$,
\begin{equation}\label{eq:energyestimate}
    \begin{split}
        & \norm{\ve(t,\cdot)}_{L^2(0,1)}^2
        +2\eps e^{2\alpha_\eps(t)}\int_0^t e^{-2\alpha_\eps(s)}
        \norm{\px \ve(s,\cdot)}^2_{L^2(0,1)}\,ds\\
        &\qquad \le 4\norm{\ve(0,\cdot)}^2_{L^2(0,1)}e^{2\alpha_\eps(t)}
        +8e^{2\alpha_\eps(t)}\int_0^t e^{-2\alpha_\eps(s)}\beta_\eps(s)\,ds,
    \end{split}
\end{equation}
where
\begin{align}\label{eq:alpha}
    \alpha_\eps(t) & = C_0\left( t+\int_0^t\left(|g_{\eps,0}(s)|+|g_{\eps,1}(s)|\right)ds\right),\\
    \label{eq:beta}
    \beta_\eps(t)&=C_0\Biggl(|g_{0,\eps}'(t)|^2+|g_{1,\eps}'(t)|^2 \\ &
    \qquad\qquad +|h_{0,\eps}(t)g_{0,\eps}(t)|^2+|h_{1,\eps}(t)g_{1,\eps}(t)|^2 \notag \\ &
    \qquad\qquad\qquad +|g_{0,\eps}(t)|^3+|g_{1,\eps}(t)|^3\Biggr),\notag
\end{align}
and $C_0>0$ is a positive  constant independent on $\eps$.

In particular, the families
\begin{align*}
    \{\ue\}_{\eps>0}, \qquad \{\sqrt{\eps}\px\ue\}_{\eps>0}
\end{align*}
are bounded in $L^\infty(0,T;L^2(0,1))$ and $L^2((0,T)\times(0,1))$, respectively.
\end{lemm}

\begin{proof}
Following \cite{Coclite:2005cr} we introduce the quantity $\te=\te(t,x)$ solving
the following elliptic problem:
\begin{equation}\label{eq:teta}
	\begin{cases}
		-\pxx\te+4\te=\ve(t,x),& x\in(0,1),\\
		\te(t,0)=\te(t,1)=0,& t\in(0,T).
	\end{cases}
\end{equation}
Our motivation for bringing in \eqref{eq:teta} comes 
from the fact that, in the case of homogeneous
boundary conditions, the quantity
\begin{equation*}
	\int_0^1\ve(\te-\pxx\te)\,dx
\end{equation*}
is conserved by \eqref{eq:DP} when $\eps=0$ (see \cite{DP:99}).
Thanks to \eqref{eq:teta} we have
\begin{equation}\label{eq:tetav}
	\begin{split}
		\norm{\te(t,\cdot)}_{H^2(0,1)} & \le \norm{\ve(t,\cdot)}_{L^2(0,1)}
		\le 4 \norm{\te(t,\cdot)}_{H^2(0,1)},\\
		\norm{\px\te(t,\cdot)}_{H^2(0,1)} & \le \norm{\px\ve(t,\cdot)}_{L^2(0,1)}
		\le 4 \norm{\px\te(t,\cdot)}_{H^2(0,1)}.
	\end{split}
\end{equation}
Indeed, squaring both sides of \eqref{eq:teta}
\begin{equation*}
	\ve^2=(\pxx\te)^2-8\te\px\te+16\te^2
\end{equation*}
and integrating over $(0,1)$
\begin{align*}
	\int_0^1\ve^2dx=&\int_0^1\Big[(\pxx\te)^2+8(\px\te)^2+16\te^2\Big]\,dx
	+8\left[\te\px\te\right]_0^1\\
	& =\int_0^1\Big[(\pxx\te)^2+8(\px\te)^2+16\te^2\Big]\,dx.
\end{align*}
Since
\begin{align*}
	\int_0^1\Big[(\pxx\te)^2+(\px\te)^2+\te^2\Big]\,dx
	& \le \int_0^1\Big[(\pxx\te)^2+8(\px\te)^2+16\te^2\Big]\,dx
	\\ & \le 16 \int_0^1\Big[(\pxx\te)^2+(\px\te)^2+\te^2\Big]\,dx,
\end{align*}
we have the first  line of \eqref{eq:tetav}. 
For the second line in  \eqref{eq:tetav}, since 
$$
\pxx\te(t,0)=\pxx\te(t,1)=0 \quad \text{(cf.~\eqref{eq:boundaryv})}, 
$$
we can argue in the same way.

We multiply \eqref{eq:DPepsweakv} by $\te-\pxx\te$ and then integrate
the result over $(0,1)$, obtaining
\begin{equation}\label{eq:enest-1}
	\begin{split}
		&\underbrace{\int_0^1\pt \ve(\te-\pxx\te)\,dx}_{A_1}
		+\underbrace{\int_0^1\pt \ome(\te-\pxx\te)\,dx}_{A_2}
		\\ &\qquad\qquad
		+\underbrace{\int_0^1\ue\px\ue(\te-\pxx\te)\,dx}_{A_3}
		+\underbrace{\int_0^1\px\Pe(\te-\pxx\te)\,dx}_{A_4}\\ &
		\qquad\qquad\qquad\qquad
		=\eps\underbrace{\int_0^1\pxx\ve(\te-\pxx\te)\,dx}_{A_5}.
	\end{split}
\end{equation}

Thanks to \eqref{eq:boundaryv} and \eqref{eq:teta},
\begin{equation}\label{eq:A1}
    \begin{split}
        A_1& =\int_0^1\pt (4\te-\pxx\te) (\te-\pxx\te)\,dx\\
        &=\int_0^1\left(4\pt\te\te-4\pt\te\pxx\te-\ptxx\te\te+\ptxx\te\pxx\te\right)\,dx\\
        &=\int_0^1\left(4\pt\te\te+5\ptx\te\px\te+\ptxx\te\pxx\te\right)\,dx
        -\left[4\pt\te\px\te+\ptx\te\te\right]_0^1\\
        &=\frac12\frac{d}{dt}\int_0^1\left(4\te^2+5(\px\te)^2+(\pxx\te)^2\right)\, dx
        = \frac12\frac{d}{dt}\norm{\te(t,\cdot)}_{\widetilde H^2(0,1)}^2,
    \end{split}
\end{equation}
where
\begin{equation*}
	\norm{f}_{\widetilde H^2(0,1)}=\sqrt{4\norm{f}_{L^2(0,1)}^2
	+5\norm{f'}_{L^2(0,1)}^2+\norm{f''}_{L^2(0,1)}^2}.
\end{equation*}

The H\"older inequality, \eqref{eq:ass}, and \eqref{eq:omega} guarantee that
\begin{equation}\label{eq:A2}
    \begin{split}
        A_2 & \le \int_0^1(\pt\ome)^2\,dx+\frac{1}{2}\int_0^1\te^2\,dx+\frac{1}{2}\int(\pxx\te)^2\,dx\\
        & \le 2\left(|g_{0,\eps}'(t)|^2+|g_{1,\eps}'(t)|^2\right)
        +\frac{1}{2}\norm{\te(t,\cdot)}_{\widetilde H^2(0,1)}^2.
    \end{split}
\end{equation}

In light of \eqref{eq:omega}, \eqref{eq:boundaryv},  and \eqref{eq:DPepsweakP},
\begin{align*}
    A_4 & = \int_0^1\left(\px\Ve\te-\px\Ve\pxx\te+\px\Ome\te-\px\Ome\pxx\te\right)\,dx\\
    & =\int_0^1\left(\px\Ve\te+\pxx\Ve\px\te+\px\Ome\te-\px\Ome\pxx\te\right)\,dx
    -\left[\px\Ve\px\te\right]_0^1\\
    & =\int_0^1\left(\px(\Ve-\pxx\Ve)\te+\px\Ome\te-\px\Ome\pxx\te\right)\,dx
    +\left[\pxx\Ve\te\right]_0^1\\
    &=\int_0^1\left(3\ue\px\ue \te-\px\Ome\pxx\te\right)\,dx.
\end{align*}
Therefore
\begin{equation}\label{eq:A34}
    \begin{split}
        A_3+A_4&=\int_0^1\left(\ue\px\ue (4\te-\pxx\te)-\px\Ome\pxx\te\right)\, dx\\
        &= \int_0^1\left(\ue\px\ue \ve-\px\Ome\pxx\te\right)\, dx\\
        & =\int_0^1\left(\ue^2\px\ue -\ue\px\ue \ome-\px\Ome\pxx\te\right)\,dx\\
        &=\int_0^1\left(\frac{\ue^2}{2}\px\ome-\px\Ome\pxx\te\right)\,dx
        +\left[\frac{\ue^3}{3}-\frac{\ue^2}{2}\ome\right]_0^1\\
        & \le \frac{|g_{0,\eps}(t)|+|g_{1,\eps}(t)|}{2}\int_0^1\ue^2\,dx
        +\frac12\int_0^1(\pxx\te)^2\,dx\\
        &\qquad +\frac12\int_0^1(\px\Ome)^2\,dx
        +\frac{|g_{0,\eps}(t)|^3+|g_{1,\eps}(t)|^3}{6}\\
        & \le c_1\Bigl(|g_{0,\eps}(t)|+|g_{1,\eps}(t)|+1\Bigr)
        \norm{\te(t,\cdot)}_{\widetilde H^2(0,1)}^2\\
        &\qquad+c_1\Bigl(|\psi_{0,\eps}(t)|^2+|\psi_{1,\eps}(t)|^2
        +|g_{0,\eps}(t)|^3+|g_{1,\eps}(t)|^3\Bigr)\\
        & \le c_1\Bigl(|g_{0,\eps}(t)|+|g_{1,\eps}(t)|+1\Bigr)
        \norm{\te(t,\cdot)}_{\widetilde H^2(0,1)}^2\\
        & \qquad+c_1\Biggl(|g_{0,\eps}'(t)|^2+|g_{1,\eps}'(t)|^2\\
        & \qquad\qquad\qquad
        + |h_{0,\eps}(t)g_{0,\eps}(t)|^2+|h_{1,\eps}(t)g_{1,\eps}(t)|^2\\
        & \qquad\qquad\qquad\qquad
        +|g_{0,\eps}(t)|^3+|g_{1,\eps}(t)|^3\Biggr),
    \end{split}
\end{equation}
for some constant $c_1>0$ that is independent on $\eps$.

By observing that \eqref{eq:boundaryv} and \eqref{eq:teta} furnish
\begin{equation*}
    \pxx\te(t,0)=\pxx\te(t,1)=0,\qquad t\in(0,T),
\end{equation*}
we achieve
\begin{equation}\label{eq:A5}
    \begin{split}
        A_5&=\eps\int_0^1\pxx (4\te-\pxx\te) (\te-\pxx\te)\,dx\\
        &=\eps\int_0^1\left(4\pxx\te\te-4(\pxx\te)^2-\pxxxx\te\te+\pxxxx\te\pxx\te\right)\,dx\\
        &=\eps\int_0^1\left(-4(\px\te)^2-4(\pxx\te)^2+\pxxx\te\px\te-(\pxxx\te)^2\right)\,dx\\
        &\qquad\qquad+\eps\left[4\px\te\te-\pxxx\te\te+\pxxx\te\pxx\te\right]_0^1\\
        &=-\eps\int_0^1\left(4(\px\te)^2+5(\pxx\te)^2+(\pxxx\te)^2\right)\,dx
        +\eps\left[\pxx\te\px\te\right]_0^1\\
        &=-\eps\norm{\px\te(t,\cdot)}_{\widetilde H^2(0,1)}^2.
    \end{split}
\end{equation}

In view of \eqref{eq:A1}, \eqref{eq:A2}, \eqref{eq:A34}, and \eqref{eq:A5}, it follows from
\eqref{eq:enest-1} that
\begin{equation}\label{eq:enest-2}
    \begin{split}
        & \frac{d}{dt} \norm{\te(t,\cdot)}_{\widetilde H^2(0,1)}^2
        +2\eps\norm{\px\te(t,\cdot)}_{\widetilde H^2(0,1)}^2\\
        & \qquad\le c_2\left(|g_{0,\eps}(t)|+|g_{1,\eps}(t)|+1\right)
        \norm{\te(t,\cdot)}_{\widetilde H^2(0,1)}^2\\
        &\qquad\qquad
        +c_2\Biggl(|g_{0,\eps}'(t)|^2+|g_{1,\eps}'(t)|^2
        \\ &\qquad\qquad\qquad\qquad +|h_{0,\eps}(t)g_{0,\eps}(t)|^2
        +|h_{1,\eps}(t)g_{1,\eps}(t)|^2
        \\ &\qquad\qquad\qquad\qquad\qquad
        +|g_{0,\eps}(t)|^3+|g_{1,\eps}(t)|^3\Biggr),
    \end{split}
\end{equation}
for some constant $c_2>0$ that is independent on $\eps$.

Using the notations introduced
in \eqref{eq:alpha} and \eqref{eq:beta}, inequality \eqref{eq:enest-2} becomes
\begin{equation*}
    \frac{d}{dt}\norm{\te(t,\cdot)}_{\widetilde H^2(0,1)}^2
    +2\eps\norm{\px\te(t,\cdot)}_{\widetilde H^2(0,1)}^2
    \le \alpha_\eps'(t)\norm{\te(t,\cdot)}_{\widetilde H^2(0,1)}^2
    +\beta_\eps(t),
\end{equation*}
and hence, thanks to the Gronwall lemma,
\begin{equation}\label{eq:enest-4}
    \begin{split}
        & \norm{\te(t,\cdot)}_{\widetilde H^2(0,1)}^2
        +2\eps e^{\alpha_\eps(t)}\int_0^t e^{-\alpha_\eps(s)}
        \norm{\px \te(s,\cdot)}^2_{\widetilde H^2(0,1)}\,ds
        \\ & \qquad\le \norm{\te(0,\cdot)}^2_{\widetilde H^2(0,1)}e^{\alpha_\eps(t)}
        +2e^{\alpha_\eps(t)}\int_0^t e^{-\alpha_\eps(s)}\beta_\eps(s)\,ds.
    \end{split}
\end{equation}
Clearly, via \eqref{eq:tetav},  the desired claim \eqref{eq:energyestimate}
follows from \eqref{eq:enest-4}.

The boundedness of the families $\{\ue\}_{\eps>0},\,\{\px\ue\}_{\eps>0}$
follows from the definition of the auxiliary variable
$\ve$ in \eqref{eq:omega} and assumption \eqref{eq:ass}.
\end{proof}

We continue with some a priori bounds that
come directly from the energy estimate
stated in Lemma \ref{lm:energyestimate}.

\begin{lemm}\label{lm:boundP}
The families $\{\Ve\}_{\eps>0}$, $\{\Pe\}_{\eps>0}$ are
both bounded in
$$
L^{\infty}(0,T;W^{2,1}(0,1)) \cap L^{\infty}(0,T;W^{1,\infty}(0,1)),
$$
In particular, these families are bounded in $L^\infty((0,T)\times(0,1))$.
\end{lemm}

\begin{proof}
To simplify the notation, let us introduce the quantity
\begin{equation*}
    f_\eps=\frac{3}{2}\ue^2+\pxx\Ome-\Ome.
\end{equation*}
From  \eqref{eq:boundaryv} and \eqref{eq:DPepsweakP},
\begin{equation*}
    -\pxx\Ve+\Ve=f_\eps,\qquad \px\Ve(t,0)=\px\Ve(t,1)=0.
\end{equation*}

Using the function
\begin{equation*}
    G(x,y)=
    \begin{cases}
        \frac{e^x+e^{-x}}{2}\frac{e^{y-1}+e^{1-y}}{e-e^{-1}},&\quad \text{if $0\le x\le y\le 1$,}\\
        \frac{e^y+e^{-y}}{2}\frac{e^{x-1}+e^{1-x}}{e-e^{-1}},&\quad \text{if $0\le y\le x\le 1$,}
    \end{cases}
\end{equation*}
which is the Green's function of the operator $1-\pxx$ on $(0,1)$ with homogenous
Neumann boundary conditions at $x=0,1$, we have the formulas
\begin{equation}
\label{eq:GMC}
\Ve(t,x)=\int_0^1 G(x,y)f_\eps(t,y)\,dy, \qquad
\px\Ve(t,x) =\int_0^1  \px G(x,y)f_\eps(t,y)\,dy.
\end{equation}
Since $G\ge 0$ and $G,\, \px G\in L^\infty((0,1)\times(0,1))$, we can
estimate as follows:
\begin{align*}
    & |\Ve(t,x)|\le\int_0^1 G(x,y)|f_\eps(t,y)|\,dy
    \le \norm{G}_{L^\infty((0,1)^2)}\norm{f(t,\cdot)}_{L^1(0,1)},\\
    & |\px\Ve(t,x)|\le\int_0^1 |\px G(x,y)||f_\eps(t,y)|\,dy
    \le \norm{\px G}_{L^\infty((0,1)^2)}
    \norm{f(t,\cdot)}_{L^1(0,1)},\\ &
    \norm{\pxx\Ve(t,\cdot)}_{L^1(0,1)}
    \le \norm{\Ve(t,\cdot)}_{L^1(0,1)}
    + \norm{f_\eps(t,\cdot)}_{L^1(0,1)}.
\end{align*}
Thanks to Lemma \ref{lm:energyestimate}, we conclude that the
desired bounds on $\{\Ve\}_{\eps>0}$ hold.

Finally, the bounds on $\{\Pe\}_{\eps>0}$ follow from the
bounds on $\{\Ve\}_{\eps>0}$ and \eqref{eq:ass}.
\end{proof}

Using the previous lemma we can bound $\ue$ and $\ve$
in $L^\infty$ (cf.~\cite[Lemma 4]{CK:DPUMI}).

\begin{lemm}\label{lm:invariance4}
For every $t\in(0,T)$,
\begin{equation*}
    \norm{\ue(t,\cdot)}_{L^\infty(0,1)}
    \le\norm{u_0}_{L^\infty(0,1)}
    +\norm{g_0}_{L^\infty(0,T)}
    +\norm{g_1}_{L^\infty(0,T)}+C_Tt,
\end{equation*}
for some constant $C_T>0$ depending on $T$ but not on $\eps$.
\end{lemm}

\begin{proof}
Due to \eqref{eq:DPepsw} and Lemma \ref{lm:boundP},
$$
\pt \ue +\ue\px\ue-\eps\pxx \ue
\le \sup\limits_{\eps>0}\norm{\px \Pe}_{L^\infty((0,T)\times(0,1))}\le C_T.
$$
Since the map
$$
f(t):=\norm{u_0}_{L^\infty(0,1)}+\norm{g_0}_{L^\infty(0,T)}
+\norm{g_1}_{L^\infty(0,T)}+C_Tt,\qquad t\in(0,T),
$$
solves the equation
$$
\frac{df}{dt}=C_T
$$
and
$$
\ue(0,x), g_0(t), g_1(t)\le f(t),\qquad (t,x)\in (0,T)\times(0,1),
$$
the comparison principle for parabolic equations implies that
$$
\ue(t,x)\le f(t),\qquad (t,x)\in (0,T)\times(0,1).
$$
This concludes the proof of the lemma.
\end{proof}

As a consequence of  Lemmas \ref{lm:energyestimate} and \ref{lm:boundP},
the second equation in \eqref{eq:DPepsw} yields
\begin{lemm}\label{lm:P_xxinfty}
The families $\{\Ve\}_{\eps>0}$, $\{\Pe\}_{\eps>0}$ are
bounded in $L^{\infty}(0,T;W^{2,\infty}(0,1))$.
\end{lemm}

Let us continue by proving the existence of  a distributional solution
to  \eqref{eq:DP}, \eqref{eq:init}, \eqref{eq:DPboundary}  satisfying \eqref{eq:DPentropy}.

\begin{lemm}\label{lm:conv}
There exists a function $u\in L^\infty((0,T)\times(0,1))$ that is a distributional
solution of \eqref{eq:DPw} and satisfies  \eqref{eq:DPentropy} in the sense of
distributions for every convex entropy $\eta\in C^2(\R)$.
\end{lemm}

We  construct a solution by passing
to the limit in a sequence $\Set{u_{\eps}}_{\eps>0}$ of viscosity
approximations \eqref{eq:DPepsw}. We use the
compensated compactness method \cite{TartarI}.

\begin{lemm}\label{lm:convul1l4}
There exists a subsequence
$\{\uek\}_{k\in\N}$ of $\{\ue\}_{\eps>0}$
and a limit function $  u\in L^\infty((0,T)\times(0,1))$
such that
\begin{equation}\label{eq:convu}
    \textrm{$\uek \to u$ a.e.~and in $L^p((0,T)\times(0,1))$, $1\le p<\infty$}.
\end{equation}
\end{lemm}

\begin{proof}
Let $\eta:\R\to\R$ be any convex $C^2$ entropy function, and let
$q:\R\to\R$ be the corresponding entropy
flux defined by $q'(u)=\eta'(u)\, u$.
By multiplying the first equation in \eqref{eq:DPepsw} with
$\eta'(\ue)$ and using the chain rule, we get
\begin{equation*}
    \pt  \eta(\ue)+\px q(\ue)
    =\underbrace{\eps \pxx \eta(\ue)}_{=:\CLea^1}
    \, \underbrace{-\eps \eta''(\ue) \left(\px  \ue\right)^2
    +\eta'(\ue) \px \Pe}_{=: \CLea^2},
\end{equation*}
where  $\CLea^1$, $\CLea^2$ are distributions. By
Lemmas \ref{lm:energyestimate}, \ref{lm:boundP}, \ref{lm:invariance4}, and \ref{lm:P_xxinfty},
\begin{equation}\label{eq:conclaim1-new}
    \begin{split}
        &\textrm{$\CLea^1\to 0$ in $H^{-1}((0,T)\times(0,1))$},\\
        &\textrm{$\CLea^2$ is uniformly bounded in $L^1((0,T)\times(0,1))$.}
    \end{split}
\end{equation}
Therefore, Murat's lemma \cite{Murat:Hneg} implies that
\begin{equation}\label{eq:GMC1}
	\text{$\left\{  \pt  \eta(\ue)+\px q(\ue)\right\}_{\eps>0}$ 
	lies in a compact subset of $\Hneg((0,T)\times(0,1))$.}
\end{equation}
The $L^\infty$ bound stated in Lemma \ref{lm:invariance4}, \eqref{eq:GMC1}, and the Tartar's compensated compactness method \cite{TartarI} give the existence of a subsequence
$\{\uek\}_{k\in\N}$ and a limit function $  u\in L^\infty((0,T)\times(0,1))$
such that \eqref{eq:convu} holds.
\end{proof}

\begin{lemm}\label{lm:convPl1l4}
We have
\begin{equation}\label{eq:convP}
	\textrm{$\Pek \to P^u$ in $L^p(0,T;W^{1,p}(0,1))$, $1\le p<\infty$,}
\end{equation}
where the sequence $\Set{\eps_k}_{k\in\N}$
and the function $u$ are constructed in
Lemma \ref{lm:convul1l4}.
\end{lemm}

\begin{proof}
Using the integral representation of $V_{\eps_k}$ stated in \eqref{eq:GMC}, Lemma \ref{lm:invariance4},
and arguing as in \cite[Theorem 3.2]{Coclite:2005cr}, we get
\begin{align*}
&\norm{\Pek-P^u}_{L^p(0,T;W^{1,p}(0,1))}\\
&\le C\left(\norm{\uek-u}_{L^p((0,T)\times(0,1))}+\norm{\psi_{\eps_k,1}-\psi_1}_{L^p(0,T)}+\norm{\psi_{\eps_k,0}-\psi_0}_{L^p(0,T)}\right),
\end{align*}
for every $1\le p<\infty$ and  some constant $C>0$ depending on $u_0, g_0,g_1$,
but not on $\eps$. Therefore Lemma \ref{lm:convul1l4} gives  \eqref{eq:convP}.
\end{proof}

\begin{proof}[Proof of Lemma \ref{lm:conv}.] Fix a test function
$\phi\in C^\infty_c([0,T)\times [0,1])$. Due to \eqref{eq:DPepsw}
\begin{align*}
    \int_0^T\int_0^1 &\left(\ue \pt \phi +
    \frac{\ue^2}{2}\px \phi- \px \Pe \phi+\eps \ue \pxx \phi\right)\, dx\,dt\\
    & +\int_{0}^1 u_{0,\eps}(x) \phi(0,x)\, dx
    +\int_0^Tg_{0,\eps}(t) \phi(t,0)\,dt
    -\int_0^Tg_{1,\eps}(t) \phi(t,1)\,dt=0.
\end{align*}
Therefore, by the assumptions on $u_{0,\eps},\,g_{0,\eps},\, g_{1,\eps}$
and Lemmas  \ref{lm:convul1l4}, \ref{lm:convPl1l4}, we conclude
that the function $u$ constructed in Lemma \ref{lm:convul1l4}
is a distributional  solution of \eqref{eq:DPw}.

Finally, we have to verify that the distributional solution $u$
satisfies the entropy inequality stated in \eqref{eq:DPentropy}.
Let $\eta\in C^2(\R)$ be a convex entropy.
The convexity of $\eta$ and  \eqref{eq:DPepsw} yield
$$
\pt \eta(\ue)+\px q(\ue)+\eta'(\ue)\px\Pe
\le\eps\pxx\eta(\ue).
$$
Therefore, \eqref{eq:DPentropy} follows from Lemmas \ref{lm:convul1l4} and \ref{lm:convPl1l4}.
\end{proof}

We are now ready for the proof of Theorem \ref{the_theorem3}.

\begin{proof}[Proof of Theorem \ref{the_theorem3}]
Since, thanks to Lemma \ref{lm:conv}, $u\in L^\infty((0,T)\times(0,1))$ is a distributional solution
of the problem
\begin{equation}\label{eq:DPwrid}
    \begin{cases}
        \pt u+u \px u=-\px P^u,&\quad (t,x)\in (0,T)\times(0,1),\\
        u(0,x)=u_{0}(x),&\quad x\in (0,1),\\
        u(t,0)=g_{0}(t),\>\>\> u(t,1)=g_{1}(t),&\quad t\in(0,T),
    \end{cases}
\end{equation}
that satisfies the entropy inequalities \eqref{eq:DPentropy}, Theorem \ref{the_theorem1} tells us
that the limit $u$ admits strong boundary traces $u^\tau_0$, $u^\tau_1$ at
$(0,T)\times \{x=0\}$, $(0,T)\times \{x=1\}$, respectively. 
Since, arguing as in Section \ref{subsec:proof-existence} (indeed 
our solution is obtained as the vanishing viscosity
limit of \eqref{eq:DPwrid}), Lemma \ref{BL_lemma} and 
the boundedness of the source term $\px P^u$ (cf.~\eqref{eq:DPsmooth}) 
imply \eqref{eq:DPentropyboundary}. Therefore, by appealing 
to Theorem \ref{the_theorem2}, the proof 
of Theorem \ref{the_theorem3} is concluded.
\end{proof}


\begin{thebibliography}{10}

\bibitem{Bardos:1979us}
C.~Bardos, A.~Y. le~Roux, and J.-C. N{\'e}d{\'e}lec.
\newblock First order quasilinear equations with boundary conditions.
\newblock {\em Comm. Partial Differential Equations}, 4(9):1017--1034, 1979.

\bibitem{C-F}
G.-Q. Chen and H.~Frid.
\newblock Divergence-measure fields and hyperbolic conservation laws.
\newblock {\em Arch. Ration. Mech. Anal.}, 147(2):89--118, 1999.

\bibitem{CHK:ParEll}
G.~M. Coclite, H.~Holden, and K.~H. Karlsen.
\newblock Wellposedness for a parabolic-elliptic system.
\newblock {\em Discrete Contin. Dyn. Syst.}, 13(3):659--682, 2005.

\bibitem{Coclite:2005cr}
G.~M. Coclite and K.~H. Karlsen.
\newblock On the well-posedness of the {D}egasperis-{P}rocesi equation.
  \newblock {\em J. Funct. Anal.},  233(1):60--91, 2006.

\bibitem{CK:DPol}
G.~M. Coclite and K.~H. Karlsen.
\newblock On the uniqueness of discontinuous solutions to the Degasperis-Procesi equation.
  \newblock {\em J. Differential Equations},  233(1):142-160, 2007.

  \bibitem{CK:DPUMI}
G.~M. Coclite and K.~H. Karlsen.
\newblock Bounded solutions for the Degasperis-Procesi equation.
  \newblock {\em Boll. Unione Mat. Ital. (9)},  1(2):439--453, 2008.

\bibitem{DP:99}
A.~Degasperis and M.~Procesi.
\newblock Asymptotic integrability.
\newblock In {\em Symmetry and perturbation theory (Rome, 1998)}, pages 23--37.
  World Sci. Publishing, River Edge, NJ, 1999.

  \bibitem{DHH:2003}
A.~Degasperis, D.~D. Holm, and A.~N.~W. Hone.
\newblock Integrable and non-integrable equations with peakons.
\newblock In {\em Nonlinear physics: theory and experiment, II (Gallipoli,
  2002)}, pages 37--43. World Sci. Publishing, River Edge, NJ, 2003.

\bibitem{DKK:2002}
A.~Degasperis, D.~D. Holm, and A.~N.~I. Khon.
\newblock A new integrable equation with peakon solutions.
\newblock {\em Teoret. Mat. Fiz.}, 133(2):170--183, 2002.

\bibitem{Dias:2002sw}
J.-P. Dias and P.~G. LeFloch.
\newblock Some existence results for conservation laws with source-term.
\newblock {\em Math. Methods Appl. Sci.}, 25(13):1149--1160, 2002.

\bibitem{D-L}
F.~Dubois and P.~LeFloch.
\newblock Boundary conditions for nonlinear hyperbolic systems of conservation laws.
\newblock {\em J. Differential Equations}, 71(1):93--122, 1988.

\bibitem{ELY:2006}
J. Escher, Y. Liu, and   Z. Yin.
\newblock Global weak solutions and blow-up structure for the Degasperis-Procesi equation.
\newblock {\em J. Funct. Anal.}, 241(2):457--485, 2006.

\bibitem{ELY:2007}
J. Escher, Y. Liu, and   Z. Yin.
\newblock Shock waves and blow-up phenomena for the periodic Degasperis-Procesi equation.
\newblock {\em Indiana Univ. Math. J.}, 56(1):87--117, 2007.

\bibitem{EY}
J. Escher and Z. Yin.
\newblock On the initial-boundary value problems for the Degasperis-Procesi equation.
\newblock {\em Phys. Lett. A}, 368(1-2):69--76, 2007.

\bibitem{Kwon}
Y.-S.~Kwon.
\newblock Well-poseness for entropy solution of multidimensional scalar
conservation laws with strong boundary condition.
\newblock {\em J. Math. Anal. Appl.}, 340(1):543--549, 2008.

\bibitem{K-V}
Y.-S. Kwon and A.~Vasseur.
\newblock Strong traces for solutions to scalar conservation laws with general flux.
\newblock {\em Arch. Ration. Mech. Anal.}, 185(3):495--513, 2007.

\bibitem{LPT}
P.-L. Lions, B.~Perthame, and E.~Tadmor.
\newblock A kinetic formulation of multidimensional scalar conservation laws and related equations.
\newblock {\em J. Amer. Math. Soc.}, 7(1):169--191, 1994.

\bibitem{LY:2006}
Y. Liu and   Z. Yin.
\newblock Global existence and blow-up phenomena for the Degasperis-Procesi equation.
\newblock {\em Comm. Math. Phys.}, 267(3):801--820, 2006.

\bibitem{Lundmark:2005kz}
H.~Lundmark.
\newblock Formation and dynamics of shock waves in the {D}egasperis-{P}rocesi
  equation.
\newblock {\em J. Nonlinear Sci.}, 17(3):169--198, 2007.

\bibitem{LS:2003}
H.~Lundmark and J.~Szmigielski.
\newblock Multi-peakon solutions of the {D}egasperis-{P}rocesi equation.
\newblock {\em Inverse Problems}, 19(6):1241--1245, 2003.

\bibitem{Murat:Hneg}
F.~Murat.
\newblock L'injection du c\^one positif de ${H}\sp{-1}$\ dans ${W}\sp{-1,\,q}$\  est compacte pour tout $q<2$.
\newblock {\em J. Math. Pures Appl. (9)}, 60(3):309--322, 1981.

\bibitem{Mustafa:DP05}
O.~G. Mustafa.
\newblock A note on the {D}egasperis-{P}rocesi equation.
\newblock {\em J. Nonlinear Math. Phys.}, 12(1):10--14, 2005.

\bibitem{O}
O. A. Ole\u\i nik.
\newblock  Discontinuous solutions of non-linear differential equations.
\newblock  {\em Amer. Math. Soc. Transl. (2)}, 26:95--172, 1963.

\bibitem{Otto:1996am}
F.~Otto.
\newblock Initial-boundary value problem for a scalar conservation law.
\newblock {\em C. R. Acad. Sci. Paris S\'er. I Math.}, 322(8):729--734, 1996.

\bibitem{Panov1}
E.~Y. Panov.
\newblock Existence of strong traces for generalized solutions of
  multidimensional scalar conservation laws.
\newblock {\em J. Hyperbolic Differ. Equ.}, 2(4):885--908, 2005.

\bibitem{Panov2}
E.~Y. Panov.
\newblock Existence of strong traces for quasi-solutions of multidimensional conservation laws.
\newblock {\em J. Hyperbolic Differ. Equ.}, 4(4):729--770, 2007.

\bibitem{Perthame}
B.~Perthame.
\newblock Uniqueness and error estimates in first order quasilinear
  conservation laws via the kinetic entropy defect measure.
\newblock {\em J. Math. Pures Appl.}, 77(10):1055--1064, 1998.

\bibitem{Perthame:2002qy}
B.~Perthame.
\newblock {\em Kinetic formulation of conservation laws}, volume~21 of {\em
  Oxford Lecture Series in Mathematics and its Applications}.
\newblock Oxford University Press, Oxford, 2002.

\bibitem{P-S}
B.~Perthame and P.~E. Souganidis.
\newblock A limiting case for velocity averaging.
\newblock {\em Ann. Sci. \'Ecole Norm. Sup. (4)}, 31(4):591--598, 1998.

\bibitem{TartarI}
L.~Tartar.
\newblock Compensated compactness and applications to partial differential   equations.
\newblock In {\em Nonlinear analysis and mechanics: Heriot-Watt Symposium, Vol.  IV},
pages 136--212. Pitman, Boston, Mass., 1979.


\bibitem{Tartar:1983ul}
L.~Tartar.
\newblock The compensated compactness method applied to systems of conservation laws.
\newblock In {\em Systems of nonlinear partial differential equations (Oxford,
  1982)}, volume 111 of {\em NATO Adv. Sci. Inst. Ser. C Math. Phys. Sci.},
  pages 263--285. Reidel, Dordrecht, 1983.

\bibitem{Vasseur}
A.~Vasseur.
\newblock Strong traces for solutions of multidimensional scalar conservation
  laws.
\newblock {\em Arch. Ration. Mech. Anal.}, 160(3):181--193, 2001.

\bibitem{Yin-DP:2003-JMAA}
Z.~Yin.
\newblock Global existence for a new periodic integrable equation.
\newblock {\em J. Math. Anal. Appl.}, 283(1):129--139, 2003.

\bibitem{Yin-DP:2003-Illinois}
Z.~Yin.
\newblock On the {C}auchy problem for an integrable equation with peakon
  solutions.
\newblock {\em Illinois J. Math.}, 47(3):649--666, 2003.

\bibitem{Yin-DP:2004-Indiana}
Z.~Yin.
\newblock Global solutions to a new integrable equation with peakons.
\newblock {\em Indiana Univ. Math. J.}, 53:1189--1210, 2004.

\bibitem{Yin-DP:2004-FunctAnal}
Z.~Yin.
\newblock Global weak solutions for a new periodic integrable equation with
  peakon solutions.
\newblock {\em J. Funct. Anal.}, 212(1):182--194, 2004.

\bibitem{Z}
J. Zhou.
\newblock Global existence of solution to an initial-boundary value problem for the Degasperis-Procesi equation.
\newblock {\em Int. J. Nonlinear Sci.}, 4(2):141--14, 2007.

\end{thebibliography}
 \end{document}